\documentclass[12pt]{amsart}
\usepackage{mathpazo}
\usepackage{hyperref}
\usepackage{bm}
\usepackage{verbatim}
\usepackage{slashbox}
\usepackage[margin=1in]{geometry}
\usepackage{array,multirow}
\usepackage{graphicx}

\usepackage[all]{xy}
\usepackage{array}

\usepackage{amssymb,amsmath,latexsym,amsthm}
\newtheorem{theorem}{Theorem}[section]
\newtheorem{lemma}[theorem]{Lemma}
\newtheorem{proposition}[theorem]{Proposition}

\newtheorem*{theorem*}{Theorem A}
\newtheorem*{theorem'}{Theorem B}
\newtheorem*{theorem"}{Theorem C}

\newtheorem{corollary}[theorem]{Corollary}
\newtheorem{predefinition}[theorem]{Definition}
\newenvironment{definition}{\begin{predefinition}\rm}{\end{predefinition}}
\newtheorem{preremark}[theorem]{Remark}
\newenvironment{remark}{\begin{preremark}\rm}{\end{preremark}}
\newtheorem{prenotation}[theorem]{Notation}
\newenvironment{notation}{\begin{prenotation}\rm}{\end{prenotation}}
\newtheorem{preexample}[theorem]{Example}
\newenvironment{example}{\begin{preexample}\rm}{\end{preexample}}
\newtheorem{preclaim}[theorem]{Claim}

\newtheorem{prequestion}[theorem]{Question}

\newtheorem{preapplication}[theorem]{Application}

\numberwithin{equation}{section}

\newcommand \ZZ {{\mathbb Z}}
\newcommand \QQ {{\mathbb Q}}

\newcommand \EE {{\mathbb E}}
\newcommand \PP {{\mathbb P}^1}

\newcommand \CC {{\mathbb C}}

\newcommand \RR {{\mathbb R}}

\newcommand \cf {{\mathfrak f}}

\newcommand \bP {{\mathbb P}}

\newcommand \cl {{\rm cl}}

\global\let\ker\undefined
\DeclareMathOperator{\ker}{Ker}

\newcommand \Sh {{\rm Sh}}

\newcommand{\cU}{\mathcal{U}}

\usepackage[usenames,dvipsnames]{color} 

\title{Data for Shimura varieties intersecting the Torelli locus}
\date{}

\author{Wanlin Li}
\address{Centre de recherches math\'ematiques,
	Universit\'e de Montr\'eal, 2920 Chemin de la tour,
	Montr\'eal (Qu\'ebec) H3T 1J4, Canada}
\email{liwanlin@crm.umontreal.ca}

\author{Elena Mantovan}
\address{Department of Mathematics, California Institute of Technology, Pasadena, CA 91125, USA}
\email{mantovan@caltech.edu}

\author{Rachel Pries}
\address{Department of Mathematics, 
Colorado State University, 
Fort Collins, CO 80523, USA}
\email{pries@math.colostate.edu}

\begin{document}

\thanks{
We would like to thank the American Institute of Mathematics for their support through the Square program. 
Li was partially supported by the Simons Collaboration on Arithmetic Geometry, Number Theory and Computation.
Pries was partially supported by NSF grant DMS-19-01819. 
}

\thanks{We would like to thank Yunqing Tang for her support and advice on this paper.
We also thank Eran Assaf, Bjorn Poonen, and John Voight for helpful conversations about class groups.
}

\begin{abstract}
For infinitely many Hurwitz spaces parametrizing cyclic covers of the projective line,
we provide a method to determine the integral PEL datum of the Shimura variety that contains the image 
of the Hurwitz space under the Torelli morphism.


MSC20 classifications: 
primary 11G15, 11G18, 11G30, 14G35, 14K10; 
secondary 11G10, 11R18, 14H10, 14H40, 14K22.
\end{abstract}












\maketitle

\section{Introduction}

\subsection{Overview}

In \cite{shimuraunitary}, Shimura studied unitary groups associated with Hermitian spaces over algebraic number fields and their maximal lattices.  In \cite{shimuraanalytic}, he developed this theory to study isomorphism classes of polarized abelian varieties and Riemann forms.
Using this, in \cite{shimuratranscend}, Shimura determined the lattice and Hermitian matrix
associated with each of six families of cyclic covers of the projective line ${\mathbb P}^1$.
The lattice and Hermitian matrix determine \emph{the integral PEL datum} of the family, 
as defined in Sections~\ref{Ssig} and \ref{Sintegraldatum}.

An important question to ask is whether Shimura's method applies to other 
families of curves or to their associated Shimura varieties.
As stated in \cite[Section 6]{shimuratranscend}, 
it is essential to have a point in the Shimura variety that represents the Jacobian of a curve, in 
order to have a self-dual lattice.  It is natural to start by studying PEL type Shimura varieties that 
are associated with families of Jacobians of degree $m$ cyclic covers of ${\mathbb P}^1$.
In this context, it is easier to identify the lattice when the class number of $\QQ(\zeta_m)$ is one.
One of Shimura's key results is that, under mild conditions, there is a unique isomorphism class
of Hermitian form for a given lattice and given signature type at the infinite places, 
\cite[Appendix, Proposition 8]{shimuratranscend}. 

In this paper, for $m$ an odd prime such that $\QQ(\zeta_m)$ has class number one,
we provide a method to determine the lattice and Hermitian matrix, 
and thus the integral PEL datum, for all positive-dimensional families of degree $m$ 
cyclic covers of ${\mathbb P}^1$, see Theorem~\ref{PELdatumbeta}. 
Our method also works for infinitely many families when $m$ is not prime.

\subsection{Shimura Data} \label{Ssig}

Our results generalize what is known about 
Shimura data of the moduli spaces of principally polarized abelian varieties to 
the context of certain unitary Shimura varieties.
In that light, the foundation of this paper is Riemann's theorem, which 
states that a complex torus is an abelian variety if and only if it is polarizable.
Thus the category of abelian varieties over $\CC$ is equivalent to the category of 
pairs $(V,\Lambda)$, where $V$ is a non-trivial $\CC$-vector space of finite dimension and 
$\Lambda$ is a lattice in $V$ that admits a Riemann form which is integral on $\Lambda$, see Section~\ref{SRiemann}.

For $g\geq 1$, let ${\mathcal A}_g$ denote the moduli space 
of principally polarized abelian varieties of dimension $g$.
Let $\Lambda = \ZZ^{2g}$ be the standard lattice in $V=\QQ^{2g}$, 
together with the standard symplectic form $\Psi: V \times V \to \QQ$, which is integral on $\Lambda$.
The points of the Siegel upper half space of dimension $g$ are complex structures $J$ on $V_\RR$ 
such that $\Psi_J=\Psi_\RR(\cdot, J\cdot)$ is a Riemann form. 
These points parametrize
principally polarized complex abelian varieties $A$ of dimension $g$, 
equipped with a trivialization $\Lambda\simeq H^1(A,\ZZ)$.
By \cite[Example~11.12]{milneShimura}, the
Shimura datum for ${\mathcal A}_g$
is given by the 
symplectic $\QQ$-vector space $(V,\Psi)$. 

We review Deligne's formulation of a Shimura datum in Definition~\ref{Ddefshimdata}.
Let $G$ denote the reductive algebraic group $G={\rm GSp}(V, \Psi)$. 
In this formulation, the Shimura datum for ${\mathcal A}_g$ is given 
by the pair $(G,{X})$, where $X$ 
is the $G(\RR)$-conjugacy class of homomorphisms 
$h: {\mathbb S} \to G_{{\mathbb R}}$, where ${\mathbb S} = {\rm Res}_{\CC/{\mathbb R}} {\mathbb G}_m$ is the Deligne torus.   

In this paper, we generalize the result about the Shimura datum of ${\mathcal A}_g$, 
by replacing ${\mathcal A}_g$ by 
certain unitary Shimura varieties which we describe in the next section.

\subsection{Results}

Consider a family of cyclic covers of the projective line ${\mathbb P}^1$, indexed by the monodromy datum 
$\gamma=(m,N,a)$, where $m$ is the degree, $N$ is the number of branch points, and $a$ is the 
inertia type, see Section~\ref{Sfamily}.
Let $Z_\gamma$ be the image of this family in ${\mathcal A}_g$ under the Torelli morphism.
In \cite{deligne-mostow}, Deligne and Mostow describe  the smallest PEL type Shimura subvariety $S_\gamma$ of ${\mathcal A}_g$ containing $Z_\gamma$.
Our goal is to compute the Shimura datum of $S_\gamma$ or, more precisely,
the \emph{integral PEL datum} of $S_\gamma$ as defined in Definition~\ref{DintPELd}.

We say that a point $P$ of $Z_\gamma$ is a \emph{distinguished point} 
(see Definition~\ref{Ddistinguished}) 
if the Jacobian $J_P$ of the curve $C_P$ represented by $P$ is a product of principally polarized abelian varieties each having complex multiplication. 
For every monodromy datum $\gamma$ such that $m$ is an odd prime, 
we prove in Proposition~\ref{Pdegenerate} that 
$Z_\gamma$ has a distinguished point $P$ such that all the abelian varieties in $J_P$
have complex multiplication by $\ZZ[\zeta_m]$.  
Here $J_P$ is the Jacobian of a curve in the boundary of the Hurwitz family; 
this curve is 
an admissible cyclic cover of a tree of projective lines, which can be explicitly computed from 
$\gamma$.

This provides a strategy to achieve the goal above. 
Suppose $m$ is an odd prime such that $\QQ(\zeta_m)$ has class number $1$. 
In other words, $m=3,5,7,11,13,17,19$.
In Theorem~\ref{PELdatumbeta}, for every monodromy datum $\gamma$ for degree $m$ covers, 
we provide a method that determines the integral PEL datum of $S_\gamma$; specifically, this method
analyzes $J_P$ to 
\begin{enumerate}
\item express $V$ as a vector space over $F=\QQ(\zeta_m)$;
find a ${\mathcal O}_F$-lattice $\Lambda \subset V$; and

\item explicitly find the Hermitian form $\langle \cdot , \cdot \rangle$ on $V$, taking integral values on $\Lambda$. 
\end{enumerate}

Our technique to find the Shimura datum of $S_\gamma$ also applies to infinitely many monodromy data 
$\gamma$ when $m$ is a composite number such that $\QQ(\zeta_m)$ has class number $1$.

\subsection{Comparison of methods}

Our approach to this problem is somewhat different from Shimura's  and may be more accessible to people 
with background in the areas of algebraic number theory and moduli spaces of curves (Hurwitz spaces).

One advantage of our approach is that it provides a straight-forward method to 
determine the lattice and the Hermitian form explicitly.  We provide many examples of this in Section~\ref{Sexample57} (for $m$ prime) and Section~\ref{Scompositem} (for $m$ composite).  
The reason is that the decomposition of the Jacobian provides a basis for the lattice as a free 
$\ZZ[\zeta_m]$-module and the Hermitian form is diagonal with respect to that basis. 
With Shimura's approach, it is necessary to find a Witt decomposition 
of the signature and lattice; this is straight-forward for Shimura's six examples, five of which
are of the form $y^m=f(x)$ for $m$ an odd prime and $f(x)$ a separable polynomial, but could cause
subtleties for more complicated families.

One drawback to our approach is that it only works 
when the CM-types of the abelian varieties 
produced in the method are simple.  
The simple property guarantees that the principal polarization on each of the abelian varieties 
is unique up to isomorphism;
this guarantees that the Hermitian form we compute is correct.  
For $m=3,5,11,17$, we show that this simple condition is automatic.
For $m=7$, we compute a 
new distinguished point to avoid the cases where a non-simple CM-type is produced, 
see Section~\ref{Sanother7}.
For $m=13,19$, we rely on Shimura's result about uniqueness of Hermitian forms.

For composite $m$, there are several complications with both methods.

In this paper, we also pay careful attention to the subtleties caused by the choice of primitive $m$th root of unity.  
See Remark~\ref{RcompareS}.

\subsection{Outline}

Section~\ref{Shurwitz} contains information about families of cyclic covers of ${\mathbb P}^1$ and information about Shimura varieties and Shimura data.  

Section~\ref{Sbackground} contains some results in algebraic number theory about the narrow class group and units of independent signs that we need for this paper and later papers.  
In Sections~\ref{Scyclo} - \ref{Ssimple}, we restrict to the case of cyclotomic fields and provide information about 
the narrow class number, the different, and simple types.

In Section~\ref{SabvarCM}, we prove a result about uniqueness of principal polarizations on abelian varieties with complex multiplication, see Propositions~\ref{CMuniqueF} and \ref{CMuniqueF2}.

Section~\ref{Sshdata} contains the main result of the paper, Theorem~\ref{PELdatumbeta}, 
which determines the integral PEL data of the unitary Shimura varieties $S_\gamma$ 
for all families $S_\gamma$ when $m$ is an odd prime such that $\QQ(\zeta_m)$ has class number $1$.

Section~\ref{Sexample57} contains examples of Theorem~\ref{PELdatumbeta}:
for all $\gamma$ when $m=3$; and for all $\gamma$ with $N=4$ branch points when $m=5$ and $m=7$. 
In Section~\ref{Sanother7}, we apply the technique of Theorem~\ref{PELdatumbeta}
to determine the integral PEL datum 
using a different kind of distinguished point that represents the Jacobian of a curve with extra automorphisms.

In Section~\ref{Scompositem}, we provide some examples when $m=4,6,10$ to 
illustrate that the technique of Theorem~\ref{PELdatumbeta}
sometimes works when $m$ is composite.

\begin{remark}
In \cite{moonen}, Moonen proved that there 
are exactly 20 special families of cyclic covers of ${\mathbb P}^1$, up to equivalence.  
The family with monodromy datum $\gamma$ is \emph{special} if 
$Z_\gamma$ is open and dense in $S_\gamma$. 
The six families in Shimura's 
paper are $M[6]$, $M[10]$, $M[8]$, $M[11]$, $M[16]$, and $M[17]$, 
where $M[n]$ denotes the $n$th row in \cite[Table 1]{moonen}.
As an application of Theorem~\ref{PELdatumbeta}, we determine the integral PEL datum for 
twelve of Moonen's special families
including the six families from \cite{shimuratranscend}. 
These provide nice examples but we emphasize that the special property 
is not necessary for either method.
\end{remark}

\section{Hurwitz families and Shimura varieties} \label{Shurwitz}

\subsection{Families of cyclic covers of the projective line} \label{Sfamily}

As in \cite[Section 2.2]{LMPT2}, we consider the universal family of $\mu_m$-covers $\psi:{\mathcal C} \to \PP$, branched at $N$ points, 
with \emph{inertia type} $a=(a(1), \ldots, a(N))$;
the data $\gamma=(m,N,a)$ is called a \emph{monodromy datum}. 
Over $\mathbb{Q}(\zeta_m)$, such a cover $\psi$ has an equation of the form
$y^m = \prod_{i=1}^N (x-t(i))^{a(i)}$, together with a choice of automorphism of order $m$ given by
$(x,y) \mapsto (x, \zeta_my)$.  
By \cite[Equation 2.2]{LMPT2}, the genus of the fibers of ${\mathcal C}$ is 
\begin{equation}\label{Egenus}
g=g_\gamma = 1 + \frac{(N-2)m - \sum_{i=1}^N {\rm gcd}(a(i), m)}{2}.
\end{equation}

The \emph{signature type} of the monodromy datum $\gamma$ is 
a function $\cf: {\rm Hom}(\QQ[\mu_m],\CC)\to \ZZ_{\geq 0}$ which we denote by
$\cf=(\cf(\tau_1), \ldots, \cf(\tau_{m-1}))$, where $\tau_n$ and $\cf(\tau_n)$ are defined as follows.
For $0 \leq n < m$, let $\tau_n:\QQ[\mu_m]\to\CC$ be given by $\tau_n(\zeta_m)=e^{2\pi I n /m}$. 
We identify
\[\ZZ/m\ZZ   = {\rm Hom}(\QQ[\mu_m],\CC) \text{  by   }
n\mapsto \tau_n.\] 
We define $\cf(\tau_n)$ to be the dimension of the eigenspace of $H^0({\mathcal C}({\mathbb C}), \Omega^1)$
on which the chosen automorphism of order $m$ acts by multiplication by $\zeta_m^n$.

There is a formula for $\cf(\tau_n)$ as follows.
For any $x\in \QQ$, let $\langle x\rangle$ denote the fractional part of $x$. 
By \cite[Lemma 2.7, \S3.2]{moonen} (or \cite{deligne-mostow}), 
\begin{eqnarray}\label{DMeqn}\label{Esign}
\cf(\tau_n) = \begin{cases} -1+\sum_{i=1}^N\langle\frac{-na(i)}{m}\rangle & \text{ if $n\not\equiv 0 \bmod m$}\\
0 & \text{ if $n\equiv 0 \bmod m$}.\end{cases}
\end{eqnarray}

Next, we describe how the inertia type $a$ and signature $\cf$
change under the action of ${\rm Aut}(\mu_m) \simeq (\ZZ/m\ZZ)^*$.
Let $\sigma_i \in {\rm Aut}(\mu_m)$ 
denote the automorphism such that $\sigma_i(\zeta_m)=\zeta_m^i$.

\begin{lemma} \label{Lchangeembed}
The automorphism $\sigma_i\in{\rm Aut}(\mu_m)$ 
takes the inertia type $a=(a_1, a_2,\dots, a_N)$ to $a'$
and the signature $\cf$ to $\cf'$ as follows: 
\[a'=i^{-1} \cdot a=(i^{-1} \cdot a_1,i^{-1} \cdot a_2, \dots, i^{-1} \cdot a_N) \text{ and }
\cf' (\tau_n) 
 =\cf(\tau_{ni^{-1}}).\]
 \end{lemma}
 
 \begin{proof}
The $j$th entry $a_j$ of the inertia type signifies that the canonical generator of inertia above the $j$th branch point is the $a_j$th power of the generator $\zeta_m$ of $\mu_m$.  The canonical 
generator of inertia is the $i^{-1} a_j$th power of the new generator $\zeta_m^i$ of $\mu_m$, so the $j$th entry of the inertia type $a'$ is $i^{-1}a_j$. 

The formula $\cf'(\tau_n) = \cf(\tau_{ni^{-1}})$ can be deduced from the formula for $a'$ and \eqref{DMeqn}.
Alternatively, it can be seen directly since the automorphism $\sigma_i$ permutes the eigenspaces for the action of the chosen automorphism on $H^0({\mathcal C}({\mathbb C}), \Omega^1)$, taking the eigenspace indexed by $n$ to the one indexed by $ni^{-1}$.
 \end{proof}
 
Let $\gamma=(m, N,a)$ be a monodromy datum.
Consider the Hurwitz family of $\mu_m$-covers of ${\mathbb P}^1$ with monodromy datum $\gamma$.
As in \cite[Definition 2.1]{LMPT2}, let $Z^0=Z^0_\gamma$ denote the image of this family in ${\mathcal A}_g$.
Let $Z=Z_\gamma$ denote the closure of $Z^0_\gamma$ in ${\mathcal A}_g$.
A point $P$ in $Z^0_\gamma$ represents the Jacobian of a smooth curve $C_P$, for which there is a 
$\mu_m$-cover $\psi_P: C_P \to {\mathbb P}^1$ with monodromy datum $\gamma$.
A point $P$ in $Z_\gamma - Z^0_\gamma$ represents the Jacobian of a singular curve $C_P$ of compact type, 
for which there is an admissible $\mu_m$-cover $\psi_P: C_P \to T$ with monodromy datum $\gamma$, where $T$ is a tree of projective lines.
 
\subsection{Complex abelian varieties} \label{SRiemann}
 
Let $g \geq 1$ and let $V = \RR^{2g}$. 
Suppose $V$ has a complex structure 
and $\Lambda \subset V$ is a lattice of rank $2g$. 
A Riemann form on a pair $(V,\Lambda)$ is an alternating pairing
$\Psi: \Lambda\times\Lambda \to \ZZ$
such that the pairing $\Psi_\RR(\cdot, \sqrt{-1} \cdot ):V\times V\to\RR$
is symmetric and positive definite.  
Every complex abelian variety of dimension $g$ is isomorphic to $V/\Lambda$ for some 
pair $(V, \Lambda)$ that admits a Riemann form.

By Riemann's Theorem, 
the category of 
polarized abelian varieties of dimension $g$ over $\CC$ is equivalent to the category of 
pairs $(V,\Lambda)$, where $V$ is a non-trivial $\CC$-vector space of dimension $g$ and 
$\Lambda$ is a lattice in $V$ together with a Riemann form.

\begin{remark}\label{Rdeligne}
Let ${\mathbb S} = {\rm Res}_{\CC/{\mathbb R}} {\mathbb G}_m$ be the Deligne torus. 
A Hodge structure on $V$ is a homomorphism $h:  {\mathbb S} \to GL_V$. 
For Hodge structures of type $\{(-1,0), (0,-1)\}$,
the $\RR$-linear map $h(i)$, called the Weil operator, defines a complex structure $J$ on $V$.
By the work of Deligne, a complex structure on $V$ gives equivalent data as 
a Hodge structure on $V$ of type $\{(-1,0), (0,-1)\}$.
\end{remark}
 
\subsection{Definition of Shimura data} \label{SdefShim}

Recall that ${\mathbb S} = {\rm Res}_{\CC/{\mathbb R}} {\mathbb G}_m$ is the Deligne torus. 

\begin{definition} \label{Ddefshimdata} \cite[Definition 11.11]{milneShimura}.
A \emph{Shimura datum} is a pair $(G,X)$ consisting of a reductive algebraic group
$G$ over $\QQ$ and a $G(\RR)$-conjugacy class of homomorphisms ${\mathbb S} \to G_\RR$ satisfying
certain conditions \cite[2.1.1(1-3)]{deligne79}.
\end{definition}
Let ${\mathbb A}$ (resp ${\mathbb A}_f$) denote the ring of adeles (resp. finite adeles) of $\QQ$.
To any Shimura datum $(G,X)$, there is an associated compatible system of smooth complex manifolds
$G(\QQ)\backslash (X\times G({\mathbb A}_f)/K) $, 
for $K$ any (sufficiently small) open compact subgroup of $G({\mathbb A}_f)$. 
These are the  {\em Shimura varieties} associated with the datum $(G,X)$.

Let $(G, X)$ is a Shimura datum and $H \subset G$ be an algebraic subgroup over $\QQ$.
Define 
\begin{equation}
Y_H=\{ h\in X \mid h: {\mathbb S} \to H_\RR\subseteq G_\RR\}.
\end{equation}
The group $H(\RR)$ acts by conjugation on the set $Y_H$
and $Y_H$ decomposes as a disjoint union of finitely many $H(\RR)$-orbits $Y\subseteq Y_H$.
We say that $(H, Y)$ is a \emph{sub Shimura datum} of $(G, X)$. 

If $(H,Y)$ is a sub Shimura datum of a Shimura datum $(G,X)$, then
the Shimura varieties associated with  $(H,Y)$ are naturally  subvarieties of those associated with $(G,X)$.

\begin{example} For $g\geq 1$, let $V=\QQ^{2g}$ and consider the lattice $\Lambda = \ZZ^{2g}$.  
Let $\Psi$ be the standard symplectic form on $V$.
Let $G_g={\rm GSp}(V, \Psi)$ be the algebraic group of symplectic similitudes over $\QQ$.
Let $X_g$ be the $G_g(\RR) $-conjugacy class of 
of homomorphisms 
$h: {\mathbb S} \to G_{g\, {\mathbb R}}$. Then $(G_g,X_g)$ is a Shimura datum.

Furthermore, a homomorphism
$h: {\mathbb S} \to  GL_{V_\RR}$ which factors through $G_{g\, \RR}\subset GL_{V_\RR}$ corresponds to 
a complex structure $J$ on $V_\RR$ such that the pairing $\psi_J(\cdot, \cdot) =\psi(\cdot, J\cdot)$ is a Riemann form.
Thus,  $X_g$ is identified with the
space of all complex structures $J$ on 
$V_\RR$ such that $(x, y) \mapsto \psi(x, Jy)$ is positive definite, see \cite[Example~11.12]{milneShimura}.
It follows that $(G_g, X_g)$ is the Shimura datum 
for ${\mathcal A}_g$, the moduli space 
of principally polarized abelian varieties of dimension $g$. That is,  there is an identification between ${\mathcal A}_g(\CC)$ and $G_g(\QQ)\backslash (X_g\times G_g({\mathbb A}_f)/K_0)$, 
for $K_0$ the stabilizer of 
$\Lambda$ in $G_g({\mathbb A}_f)$, i.e., $K_0={\rm GSp}(\Lambda, \Psi)(\hat{\ZZ})$ for $\hat{\ZZ}\subset {\mathbb A}_f$ the profinite completion of $\ZZ$.
\end{example}

\begin{definition}
A Shimura datum is of \emph{Hodge type} if it is a sub Shimura datum of $(G_g,X_g)$, for some $g \geq 1$.
\end{definition}

Shimura varieties of Hodge type are realized as subvarieties of ${\mathcal A}_g$.

\subsection{Shimura data of PEL type}\label{SdefPEL}

We consider a subclass of Shimura data of Hodge type.

As in \cite[Section 4]{kottwitz92},
a {\em Shimura datum of PEL type} is a Shimura datum arising from a {\em PEL datum} $(B,*,V,\langle, \rangle, h)$ defined as follows.
Let $B$ be a semisimple finite dimensional $\QQ$-algebra, with a positive involution $*$.
Let $V$ be a finitely generated $B$-module, together with a non-degenerate skew-symmetric $*$-Hermitian form 
$\langle\cdot,\cdot\rangle$ on $V$.  
Let $h:\CC\to {\rm End}_{B\otimes \RR}(V_\RR)$
be an $\RR$-algebra homomorphism satisfying $\langle h(z) v , w\rangle=\langle v,h(\bar{z})w\rangle$, for all $v,w\in V_\RR$ and $z\in\CC$. Assume that the symmetric bilinear form
 $\langle\cdot, h(i)\cdot\rangle$ on $V$ is positive definite. 
 
Consider the algebraic group $H$ over $\QQ$ whose points over a $\QQ$-algebra $R$ are given by 
\begin{equation} \label{EdefG}
H(R)=\{x\in {\rm End}_{B\otimes R}(V_R) \mid xx^*\in R^\times\}.
\end{equation}
Then the restriction of $h$ to $\CC^\times$ defines a homomorphism of real algebraic groups ${\mathbb S} \to H_\RR$.
Define $Y$ to be the $H(\RR)$-conjugacy class of $h$.
The pair $(H, Y)$ is a Shimura datum;
it is the Shimura datum associated to the PEL datum $(B,*,V,\langle, \rangle, h)$.

By definition, $H=GL_B(V)\cap {\rm GSp}(V,\langle, \rangle)$. Thus $(H,Y)$ can be realized as a sub Shimura datum of $(G_g,X_g)$ for $g=\dim_\QQ(V)$; hence $(H,Y)$ is a Shimura datum of Hodge type.  

To any PEL datum $(B,*,V,\langle, \rangle, h)$,
there is an associated PEL moduli space ${\rm Sh}(H, Y)={\rm Sh}(B,*,V,\langle, \rangle, h)$, which is a subvariety of ${\mathcal A}_g$, see \cite[Section 5]{kottwitz92}.
It is defined over the reflex field of $(H,Y)$ which is a finite extension of $\QQ$.
By construction, ${\rm Sh}(H, Y)$ only depends on $(H_{\mathbb A}, Y)$. 
When the Hasse principle fails, ${\rm {Sh}}(H,Y)({\mathbb C})$ is the disjoint union of finitely many Shimura varieties, associated with all $\QQ$-forms of $H_{\mathbb A}$, see \cite[Sections~7-8]{kottwitz92}.

\subsection{Integral PEL data} \label{Sintegraldatum}

Let $(B,*,V,\langle, \rangle, h)$ be a PEL datum, and $(H,Y)$ the associated Shimura datum.

\begin{definition} \label{DintPELd}
An {\em integral PEL datum} of the PEL datum $(B,*,V,\langle, \rangle, h)$
consists of a tuple $({\mathcal O}_B,*,\Lambda, \langle, \rangle, h)$ where: 
$\mathcal{O}_B$ is an order of $B$ that is $*$-stable; and $\Lambda$ is a lattice of $V$ 
that is $\mathcal{O}_B$-stable; such that the $*$-Hermitian 
form $\langle \cdot, \cdot\rangle$ is integral on $\Lambda$.
\end{definition}

Sometimes an {\em integral PEL datum} is called an {\em integral model} of $(B,*,V,\langle, \rangle, h)$.

A rational prime $p$ is  {\em good} for an integral PEL datum if $\mathcal{O}_B$ is maximal at $p$ and $\Lambda$ is self-dual at $p$.  All but finitely many primes are good, \cite[Definition 1.4.1.1]{lan}. 
If $p$ is a good prime for an integral model of $(B,*,V,\langle, \rangle, h)$, then the PEL type modular variety ${\rm Sh}(H,Y)$ has good reduction at $p$. 

To an integral PEL datum $({\mathcal O}_B,*,\Lambda, \langle, \rangle, h)$, there is an associated (canonical)  integral model ${\rm Sh}({\mathcal O}_B,*,\Lambda, \langle, \rangle, h)$ of ${\rm Sh}(H,Y)$, defined over the ring of integers of the reflex field of $(H,Y)$ localized at good primes, \cite[Theorem  1.4.1.11]{lan}. 
If $p$ is a good prime for the integral PEL datum, then ${\rm Sh}({\mathcal O}_B,*,\Lambda, \langle, \rangle, h)$ has good reduction at $p$. 
\footnote{More precisely, for any level $K$ which is neat and hyperspecial at $p$, the scheme has good reduction at $p$.}

\subsection{Shimura varieties of Deligne and Mostow} \label{Sshimura}

Let $\gamma=(m,N,a)$ be a monodromy datum, and let $Z=Z_\gamma$ be as in Section \ref{Sfamily}.  In \cite{deligne-mostow}, Deligne and Mostow observe that there is a smallest subvariety of PEL type $S_\gamma$ containing $Z_\gamma$. 
The Shimura datum $(H,Y)=(H_\gamma,Y_\gamma)$ of $S_\gamma$ is associated with a PEL datum $(B,*,V,\langle, \rangle, h)$ as in Section \ref{SdefPEL}, satisfying the following conditions:
\begin{itemize}
\item $B=\QQ[\mu_m]$ is the group algebra of $\mu_m$ over $\QQ$;
\item $*$ is the involution on $\QQ[\mu_m]$ induced by the inverse map on $\mu_m$; 
\item $\dim_\QQ (V)={2g_\gamma}$.
\end{itemize}
Our goal is to compute the integral PEL datum of $S_\gamma$.

In \cite{deligne-mostow}, Deligne and Mostow compute the signature type $\cf=\cf_\gamma$ 
associated with the monodromy datum $\gamma$, see \eqref{DMeqn}. 
A signature type $\cf:(\ZZ/m \ZZ)^*\to [0,g]\cap \ZZ$ uniquely determines up to isomorphism: the structure of $V$ as $B$-module; {the complex structure on $V_\RR$}; 
and the group $H_\RR$. 

Hence, our goal is to explicitly compute the skew Hermitian form 
$\langle\cdot ,\cdot\rangle$ on $V$. 
More precisely, 
let $\mathcal{O}_B=\ZZ[\mu_m]$, which is a $*$-stable order of $B$.
We compute a  $\mathcal{O}_B$-stable lattice $\Lambda$ of $V$ and a skew Hermitian form 
$\langle\cdot ,\cdot\rangle$ on $V$ which is integral on $\Lambda$. Furthermore, for $m$ prime, the lattice $\Lambda$ is self-dual away from $m$.
That is, we compute an integral PEL datum $({\mathcal O}_B,*,\Lambda, \langle, \rangle, h)$ for $S_\gamma$, for which all primes except $m$ are good.


\subsection{Strategy}

If $Z$ is an irreducible algebraic subvariety of ${\mathcal A}_g$, 
then there is a unique smallest subvariety of Hodge type $S$ containing $Z$.  
The subgroup $H\subset G$ in the Shimura datum of $S$ is not uniquely determined, 
but it can be taken to be the generic Mumford--Tate group on $Z$;
in particular, it can be computed at a point of $Z$,
see \cite[Sections 2.8 and 2.9, also 3.2]{moonenLinearity}.
For a monodromy datum $\gamma=(m,N,a)$, 
it thus suffices to compute the Shimura datum at a distinguished point $P$ of $Z_\gamma$  (Definition \ref{Ddistinguished});
when $m$ is an odd prime, we prove that such a point $P$ exists in Proposition \ref{Pdegenerate}.
Write the first Betti cohomology of the curve $C_P$ as $V=H^1(C_P({\mathbb C}), {\mathbb Q})$.

Assume $m$ is prime. We write $F=\QQ(\zeta_m)$ for the $m$-th cyclotomic field, and ${\mathcal O}_F=\ZZ[\zeta_m]$ for its ring of integers. From the signature type $\cf=\cf_\gamma$, we deduce that 
$V\simeq F^r$ as $\QQ[\mu_m]$-modules, where the action of $\QQ[\mu_m]$ on $F^r$ is induced by the natural homomorphism $\QQ[\mu_m]\to F=\QQ(\zeta_m)$,
and $r=2g/(m-1)$, for $g=g_\gamma$ as in (\ref{Egenus}).  Also,  
$r= \cf(\tau)+\cf(\tau^*)$ for any $\tau: F\to \CC$ by (\ref{Esign});   
the complex structure on $V_\RR$ is given as $V_\RR\simeq (F\otimes_\QQ \RR)^r\simeq \oplus_\tau \CC^{\cf(\tau)}$.
The standard lattice $\Lambda\simeq{\mathcal O}_F^r$ defines a lattice $\Lambda\subset V$.  

In Theorem \ref{PELdatumbeta}, we construct an explicit Hermitian form $\psi$ on $V$, with signature $\cf=\cf_\gamma$, which is integral on $\Lambda$.  
When $F$ has class number $1$, for $m=3,5,7,11,17$, 
we give a new proof that there is a unique possibility for the lattice and Hermitian form for $S_\gamma$, 
see Remark~\ref{RgoodM}. 
Thus we have computed the integral PEL datum $(\ZZ[\mu_m], *, \Lambda, \psi, h)$
defining the Shimura datum $(H_\gamma,Y_\gamma)$.  
of $S_\gamma$.

\subsection{Complex multiplication} \label{SfirstCM}

Let $L$ be a CM-field and let $L_0$ be the maximal totally real subfield of $L$. 
Let $n=[L_0:\QQ]$; hence, $[L:\QQ]=2n$.
A CM-type of $L$ is an ordered set $\Phi$ of distinct embeddings 
$\phi_i:L \hookrightarrow \CC$, for $1 \leq i \leq n$, no two of which are complex conjugate.
A CM-type of $L$ is called simple if it is not induced from the CM-type of a proper CM-subfield of $L$.

Let $A$ be a complex torus such that $L \subset {\rm End}(A) \otimes \QQ$.
We say that $A$ is of CM-type $(L, \Phi)$ if ${\rm dim}(A)=n$ and the complex representation of 
$L$ on ${\rm Lie}(A)$ is isomorphic to $\sum_{\phi \in \Phi} \phi$.
If, in addition, ${\rm End}(A) \simeq {\mathcal O}_L$, we say $A$ has type $({\mathcal O}_L, \Phi)$.

By \cite[Lemma 3.1]{LMPT1},  when $N=3$, the Jacobian of a $\mu_m$-cover $\psi:{\mathcal C} \to \PP$, branched at $3$ points, with inertia type $a=(a_1,a_2,a_3)$ and signature type $\cf$,
has complex multiplication by $\prod_{d}\QQ(\zeta_d)$, 
where $d$ satisfies $1<d \mid m$, and $d\nmid  a(i)$ for  $1\leq i\leq 3$. 
Its 
CM-type $\Phi$ satisfies
\begin{equation} \label{EdefCMtype}
\phi\in\Phi \text{ if and only if } \cf(\phi)>0.
\end{equation}
The maximal order in $\prod_{d}\QQ(\zeta_d)$ is $\prod_{d}\ZZ[\zeta_d]$.
The image of the group ring $\ZZ[\mu_m]$ in $\prod_{d}\QQ(\zeta_d)$ has finite index inside the maximal order. 
This index is 1 if and only if there is a unique integer $d$ satisfying $1<d \mid m$, and $d\nmid  a(i)$ for  $1\leq i\leq 3$ (for example, if $m$ is prime). 

The next result is immediate from Lemma~\ref{Lchangeembed} and \eqref{EdefCMtype}.

\begin{lemma} \label{Lchangeembed2}
If $N=3$, then $\sigma_i\in{\rm Aut}(\mu_m)$ takes the CM-type $\Phi$ to 
 $\Phi'=i \cdot \Phi =\{\tau_{ni} \mid \tau_n \in\Phi\}$.
\end{lemma}

For any CM-type $(L,\Phi)$, there exists a unique $0$-dimensional
PEL type moduli space $\Sh=\Sh(L,\Phi)$ whose points represent abelian varieties with complex multiplication of 
type $({\mathcal O}_L,\Phi)$.  (When $L=\QQ(\zeta_m)$, the points of $\Sh$ represent the Jacobians of 
curves for which there is a $\mu_m$-cover of ${\mathbb P}^1$ branched at $N=3$ points with CM-type $\Phi$.)
The signature type $\cf$ of $\Sh$ is given as: 
\begin{equation}\label{CMtype_sign} 
\text{ for each }\phi: L\to \CC: \cf(\phi)=1\text{ if }\phi\in\Phi \text{ and } \cf(\phi)=0\text{ if }\phi\notin\Phi.
\end{equation}
For any signature type $\cf$ on $L$ taking values in $\{0,1\}$, 
there is a unique CM-type $\Phi$ of $L$ compatible with $\cf$.

\section{Some background from algebraic number theory} \label{Sbackground}

\subsection{Class group and narrow class group} \label{Snarrowclass}

For a number field $L$: 
let $\mathcal{O}_L$ denote its ring of integers;
let $\mathcal{U}_L$ denote the units in $\mathcal{O}_L$;
and let $D_{L/\QQ}$ denote the different of $L$ over $\QQ$.

Let $\cl_L$ be the class group of $L$.  Recall that $\cl_L=I_L/P_L$, where 
$I_L$ is the group of non-zero fractional ideals
and $P_L$ is the subgroup of non-zero principal fractional ideals. 
Let $h_L =|\cl_L|$ be the class number of $L$.

An element $\alpha \in L$ is totally positive if it is positive under every real embedding of $L$.
Let $\mathcal{U}_L^+\subset \mathcal{U}_L$ be the subgroup of totally positive units.
Let $P_L^+ \subset P_L$ be the subgroup of principal ideals generated by a totally positive element.
The narrow class group of $L$ is $\cl_L^+=I_L/P^+_L$ and the narrow class number is $h_L^+=|\cl_L^+|$.
There is a surjection \[\nu_L: \cl_L^+ \to \cl_L.\]

\subsection{Units of independent signs} \label{Sindepsigns}

Let $L$ be a CM-field and $L_0$ its maximal totally real subfield. 
Let $n=[L_0:\QQ]$; hence, $[L:\QQ]=2n$.
By \cite[Theorem 4.10]{washington}, $h_{L_0}$ divides $h_L$.

We fix an ordering $\tau_1, \ldots, \tau_n$ of the $n$ real embeddings $L_0 \hookrightarrow \RR$.  
If $L_0/\QQ$ is Galois, we identify these with the elements of ${\rm Gal}(L_0/\QQ)$.
Consider the group homomorphism
\begin{equation} \label{Edefrho}
\rho_{L_0}:  \mathcal{U}_{L_0} \to \{\pm 1\}^{n}, \ \rho_{L_0}(u)= (\tau_i(u)/|\tau_i(u)|)_{1 \leq i \leq n} \ \mbox{ for } 
u \in \mathcal{U}_{L_0}.
\end{equation}

We say that $L_0$ has {\it units of independent signs} if, for each real embedding $\tau$, 
there is a unit which is negative under $\tau$ but positive under all other real embeddings, see \cite[Definition~12.1]{bookCH}.
This is equivalent to saying that $\rho_{L_0}$ is surjective or that $\nu_{L_0}: \cl^+_{L_0} \to \cl_{L_0}$ is an isomorphism, see \cite[Lemma~11.2]{bookCH}.

We say that $L_0$ has {\it units of almost independent signs} if every unit in ${\mathcal U}_{L_0}$ is negative under an even number of real embeddings and, for every pair of real embeddings, there is a unit which is negative under exactly the two embeddings in that pair, see \cite[Definition 12.13]{bookCH}.
This condition is equivalent to $|{\rm ker}(\nu_{L_0})|=2$ or  $|{\rm coker}(\rho_{L_0})|=2$, see \cite[Lemma 11.2]{bookCH}. 

\subsection{Norms and the Hasse unit index}

Consider the norm map $N: \mathcal{U}_L \to \mathcal{U}_{L_0}$ 
given by $N(y)= N_{L/L_0}(y)$ for $y \in \mathcal{U}_L$.
By Dirichlet's Unit Theorem, $[\mathcal{U}_{L_0}:\mathcal{U}_{L_0}^2]=2^n$.

\begin{lemma}\label{chain}
Suppose $L_0$ is a totally real field, and $L/L_0$ is a CM-extension.
Then
\begin{equation}
\mathcal{U}_{L_0}^2 \subseteq N(\mathcal{U}_L) \subseteq \mathcal{U}_{L_0}^+ \subseteq \mathcal{U}_{L_0}.
\end{equation}
\end{lemma}

\begin{proof}
By definition, $\mathcal{U}_{L_0}^+ \subseteq \mathcal{U}_{L_0}$.
All elements in $N(\mathcal{U}_{L})$ 
are totally positive units because $L$ is a CM-field, 
quadratic over its totally real subfield $L_0$.  So $N(\mathcal{U}_{L}) 
\subseteq \mathcal{U}^+_{L_0}$.
Also, $\mathcal{U}_{L_0}^2=N(\mathcal{U}_{L_0})$, hence $\mathcal{U}_{L_0}^2 \subset N(\mathcal{U}_L)$.
\end{proof}

Let $\mu_L$ be the group of roots of unity of $L$.

\begin{definition} \label{Dhasseunitindex}
The {\em Hasse unit index} of a CM-extension $L/L_0$ is $Q(L)=[\mathcal{U}_L:\mu_L \mathcal{U}_{L_0}]$.
\end{definition}

By \cite[Theorem 4.12]{washington},
$Q(L)=1$ or $Q(L)=2$. 
Since $\ker(N)=\mu_L$, it follows that $\mathcal{U}_L=\mu_L \mathcal{U}_{L_0}$ if and only if $N(\mathcal{U}_L) = N(\mathcal{U}_{L_0})$. 
Also
$N(\mathcal{U}_{L_0}) = \mathcal{U}_{L_0}^2$.  Thus 
\begin{equation} \label{EQexact}
Q(L)=[N(\mathcal{U}_L):\mathcal{U}_{L_0}^2].
\end{equation}

Let $t$ be the number of finite primes ramified in $L/L_0$.

\begin{lemma}\label{Qtype}
Suppose $L$ has odd class number.
Let $L_0$ be a totally real field, and $L/L_0$ a CM-extension.
Then $Q(L)=1$ if and only if $t=1$; and $Q(L)=2$ if and only if $t=0$.
\end{lemma}

\begin{proof}
The material in \cite[Chapter 13]{bookCH} is stated in terms of the type of the CM-extension $L/L_0$.
By a theorem of Kummer, see \cite[Theorem 13.4]{bookCH}, $L/L_0$ has type I (resp.\ II) if and only if $Q(L)=1$
(resp.\ $Q(L)=2$).
The result is then immediate from \cite[page 73]{bookCH}.
\end{proof}

We summarize the result from this section that we need in this paper and later papers.

\begin{lemma}\label{N=R} 
Suppose $L$ is a CM-field with maximal totally real subfield $L_0$.
\begin{enumerate}
\item 
Then $L_0$ has units of independent signs if and only if $Q(L)[\mathcal{U}^+_{L_0}:N(\mathcal{U}_L)]=1$.
\item Also $L_0$ has units of almost independent signs if and only if ${Q(L)}[\mathcal{U}^+_{L_0}:N(\mathcal{U}_L)]=2$.
\end{enumerate}
\end{lemma}

\begin{proof}
\begin{enumerate}
\item 
By \cite[Lemma 12.2]{bookCH}, $L_0$ has units of independent signs 
if and only if every element of ${\mathcal U}_{L_0}^+$ is a square.
The result then follows from Lemma~\ref{chain}.
\item By definition, $L_0$ has units of almost independent signs if and only if $[\mathcal{U}_{L_0}:\mathcal{U}_{L_0}^+]=2^{n-1}$. Since 
$[\mathcal{U}_{L_0}:\mathcal{U}_{L_0}^2]=2^n$, this is equivalent to $[\mathcal{U}_{L_0}^+:\mathcal{U}_{L_0}^2]=2$.
\end{enumerate}
 \end{proof}

\subsection{Cyclotomic fields} \label{Scyclo}

Let $m$ be a positive integer and let $\zeta_m = e^{2 \pi I/m} \in \CC$.
Let $\mu_m \subset {\mathbb C}$ be the group of $m$-th roots of unity.
Let $\QQ[\mu_m]$ be the group algebra of $\mu_m$ over $\QQ$.

The cyclotomic field $F=\QQ(\zeta_m)$ is a CM-field over $\QQ$
with maximal totally real subfield $F_0=\QQ(\zeta_m + \zeta_m^{-1})$. 
The degree of $F_0$ over $\QQ$ is $n=\phi(m)/2$.

\begin{notation} \label{Ndefsigma}
The choice of $\zeta_m$ fixes an embedding $\sigma_1:F\hookrightarrow\CC$.
For $1 \leq i < m$, with ${\rm gcd}(i,m)=1$,
let $\sigma_i$ be the embedding $F \hookrightarrow \CC$ (or automorphism in ${\rm Gal}(F/\QQ)$) 
determined by $\sigma_i(\zeta_m)=\zeta_m^i$.
For $x \in F$, let $\bar{x}=\sigma_{m-1}(x)$ denote its complex conjugate.
\end{notation}

If $m=p^r$ is a prime power, then 
$F/F_0$ is ramified at the unique prime of $F_0$ above $p$ and at the $n$ infinite primes of $F_0$ and 
is unramified at all other primes, \cite[Proposition 2.15]{washington}.

\begin{lemma} \label{Lgoodm}
Let $F=\QQ(\zeta_m)$ and $F_0=\QQ(\zeta_m + \zeta_m^{-1})$.  Suppose $F$ has class number $1$.
\begin{enumerate}
\item If $m$ is a prime power (or twice a prime power), then 
$F_0$ has narrow class number $1$ and thus has units of independent signs.  
The complete list of these $m$ is: 
\begin{eqnarray*} \label{Egoodm}
m & = & 1,2,4,8,16,32;\\
m & = & 3,5,7,9,11,13,17,19,25,27; \text {or twice a value of $m$ on this row.}
\end{eqnarray*} 
\item 
If $m$ is not a prime power (or twice a prime power),
then $F_0$ has narrow class number $2$ and thus has units of almost independent signs.
The complete list of these $m$ is:
\begin{eqnarray*} \label{Egoodm2}
m & = &12,16,20,24,28,36,40,44,48,60,84;\\
m & = &15,21,33,35,45;\text {or twice a value of $m$ on this row.}
\end{eqnarray*}
\end{enumerate}
\end{lemma}

\begin{proof}
The complete list of $m$ such that $h_F=1$ is well-known; it is the union of the lists in parts (1) and (2).
For these $m$, since $h_F=1$, also $h_{F_0}=1$.   

\begin{enumerate}
\item
By a result originally due to Hasse, see \cite[Corollary 3.9]{bookCH}, since $h_F$ and $h_{F_0}$ are odd, then $F_0$ has units of independent signs if and only 
if $F/F_0$ is ramified at exactly one finite prime.  
This happens if and only if $m=p^r$ or $m=2 \cdot p^r$, for some prime $p$.

\item 

Since $h_F$ and $h_{F_0}$ are odd, then $F_0$ has units of almost independent signs if and only $F/F_0$ is not ramified at any finite prime, see \cite[Corollary 13.10]{bookCH}. 
This happens if $m$ is not a prime power (or twice a prime power).
\end{enumerate}
\end{proof}

\subsection{A generator for the different of cyclotomic fields} \label{Sdifferent}

\begin{lemma}\label{beta_lemma}
Let $F=\QQ(\zeta_m)$.
The element $\beta_0$ below generates $D_{F/\QQ}$ and 
$\beta_0=-\overline{\beta}_0$. 
\begin{enumerate}
\item If $m$ is an odd prime, then 
\[\beta_0 = m/(\zeta_m^{(m+1)/2} - \zeta_m^{(m-1)/2}).\]
\item If $m=2^k$ with $k \geq 2$, then $\beta_0 = -2^{k-1} i$.
\item If $m=3^k$, then $\beta_0 = - 3^{k-1} \sqrt{3} i$.
\item If $m=pq$ with $p,q$ distinct odd primes, then 
\[\beta_0=
m\frac{\zeta_m^{(m+1)/2} - \zeta_m^{(m-1)/2}}{(\zeta_m^{q(p+1)/2} - \zeta_m^{q(p-1)/2})(\zeta_m^{p(q+1)/2} - \zeta_m^{p(q-1)/2}) }.
\]
\end{enumerate}
\end{lemma}

\begin{proof}
In each case, $\beta_0$ is on the imaginary axis so condition (2) is satisfied. 
The different $D_{F/\QQ}$ is generated by $\langle c_m'(\zeta_m) \rangle$,
see \cite[Theorem 4.3]{conraddifferent}.

Let $c_m(x)$ denote the $m$th cyclotomic polynomial.
To show that $\beta_0$ generates $D_{F/\QQ}$, it suffices to show that it is an associate of 
$\langle c_m'(\zeta_m) \rangle$ in $\mathcal{O}_F$.
\begin{enumerate}
\item When $m$ is an odd prime, then $c_m(x)=(x^m-1)/(x-1)$.
Then \[c_m'(x) = ((x-1) mx^{m-1} - (x^m-1))/(x-1)^2.\]  So $c_m'(\zeta_m) = m \zeta_m^{-1}/(\zeta_m-1)$. 
So
 $\beta_0$ from part (1) is an associate of $c_m'(\zeta_m)$ in $\mathcal{O}_F$.
 \item If $m=2^k$, then $c_m(x)=x^{2^{k-1}} +1$.  So $c'_m(\zeta_m)= 2^{k-1} \zeta_m^{2^{k-1}-1} = - 2^{k-1}/\zeta_m$.
So
 $\beta_0$ from part (2) is an associate of $c_m'(\zeta_m)$ in $\mathcal{O}_F$.
 \item If $m=3^k$, then $c_m(x)=x^{2\cdot 3^{k-1}} + x^{3^{k-1}} + 1$.
 So $c_m'(\zeta_m)= 3^{k-1} \zeta_m^{3^{k-1}-1} \sqrt{-3}$.
 So $\beta_0$ from part (3) is an associate of $c_m'(\zeta_m)$ in $\mathcal{O}_F$. 
 \item In this case,
$c_m(x)=(x^m-1)(x-1)/((x^p-1)(x^q-1))$.
The element $\beta_0$ from part (4) is an associate of
$c_m'(\zeta_m) = m \zeta_m^{-1}(\zeta_m-1)
/(\zeta^p_m-1)(\zeta^q_m-1)$.
 \end{enumerate}
\end{proof}

\subsection{Simple types} \label{Ssimple}
 
Let $F=\QQ(\zeta_m)$ and 
recall that ${\rm Gal}(F/\QQ) \simeq (\ZZ/m\ZZ)^*$.
A CM-type $\Phi$ of $F$ is simple, 
if it is not induced from any CM-field $F' \subset F$.

Let $H \subset {\rm Gal}(F/\QQ)$ and let $F^H$ be the fixed field of $F$ under $H$.
If $F^H$ is a CM-field, then $\Phi$ is induced from $F^H$ if and only if it is a union of cosets of $H$.

\begin{lemma} \label{LFermatsimple}
Suppose $m=4$ or $m$ is a Fermat prime or twice a Fermat prime.  Then $\Phi$ is simple.
\end{lemma}

We will use Lemma~\ref{LFermatsimple} for $m=4$ and $m=3,5,17$ and $m=6,10$.

\begin{proof}
For these values of $m$, the field $F$ contains no proper non-trivial CM-fields.
\end{proof}

\begin{lemma} \label{Lnotsimple} 
Let $m > 3$ be prime.
Let $\gamma=(m, 3, a)$ be a monodromy datum with inertia type $a=(1,a_2,a_3)$. 
Let $\Phi$ be the CM-type of $F$ corresponding to $a$.
Then $\Phi$ is simple if and only if $a \not =(1, x, x^2)$ for some $x \in (\ZZ/m\ZZ)^*$ with order $3$.
\end{lemma}

\begin{proof}
This follows from \cite[Theorem 2]{KRinv}.  See also \cite[Theorem 6.2]{LangCM}.
\end{proof}

For example, when $m=7$ and $a=(1,a_2,a_3)$, then $\Phi$ is simple unless $a=(1,2,4)$.

By a direct computation, one can check that $\Phi$ is always simple when $m=25$ or $m=27$. 

\section{Principal polarizations of CM abelian varieties} \label{SabvarCM}

Suppose $L$ is a CM-field and $L_0$ is its maximal totally real subfield.
In this section, we study principal polarizations on abelian varieties 
having complex multiplication by $L$. 

Section~\ref{part1} contains background about complex tori with complex multiplication.
In Section~\ref{Spptori}, we study principal polarizations on CM-abelian varieties, following
work of Van Wamelen.
In Section~\ref{part2}, we prove a result about existence and uniqueness of principal polarizations for 
abelian varieties with simple CM-type when $L$ has class number 1.
In Section~\ref{Spart2cyc}, we specialize to the case that 
$L$ is a cyclotomic field with class number $1$.

\subsection{Complex tori} \label{Scmtheory}\label{part1}

We review some complex multiplication theory, following Lang in \cite[Chapter 1]{LangCM}.
Let $A$ be a complex torus such that $L \subset {\rm End}(A) \otimes \QQ$.
We say that $A$ is of type $(L, \Phi)$ if the complex representation of 
${\rm End}(A) \otimes \QQ$ is isomorphic to $\sum_{\phi \in \Phi} \phi$.
If, in addition, ${\rm End}(A) \simeq {\mathcal O}_L$, we say $A$ has type $({\mathcal O}_L, \Phi)$.

\begin{theorem} \cite[Chapter 1: Theorems 3.3, 3.5, 4.1, 4.2]{LangCM} \label{Tlang}
\begin{enumerate}
\item If $\mathfrak{a}$ is a lattice in $L$ and $\Phi$ is a CM-type of $L$, then 
$\CC^n/\Phi(\mathfrak{a})$ is a complex torus of type $(L, \Phi)$.
\item If $A$ is a complex torus of type $(L, \Phi)$, then there exists a lattice $a$ of $L$ such that $A \simeq \CC^n/\Phi(\mathfrak{a})$.
\item  If $\Phi$ is a simple type and $\mathfrak{a}$ is a fractional ideal of $L$, then ${\rm End}(\CC^n/\Phi(\mathfrak{a})) \simeq {\mathcal O}_L$.
\item If $\Phi$ is a simple type and $\mathfrak{a},\mathfrak{b}$ are fractional ideals of $L$, 
then $\CC^n/\Phi(\mathfrak{a}) \simeq \CC^n/\Phi(\mathfrak{b})$ 
if and only if $\mathfrak{a}$ and $\mathfrak{b}$ are in the same ideal class.
\end{enumerate}
\end{theorem}

In particular, if $(L,\Phi)$ is a simple CM-type, then the set of isomorphism classes of complex tori of type 
$({\mathcal O}_L,\Phi)$ is in bijection with the class group of $L$.

Furthermore, by \cite[Chapter 1, Theorem 4.5]{LangCM},
every (admissible, non-degenerate) Riemann form on $\CC^n/\Phi(\mathfrak{a})$
is given by 
\begin{equation} \label{Ehermdef}
{\mathbb E}(\Phi(x), \Phi(y))={\rm tr}_{L/\QQ} (\xi x \bar{y}), \text{ for } x,y \in L,
\end{equation}
for some $\xi$ such that $L=L_0(\xi)$, $\xi^2 \in L_0$ is totally negative, 
and ${\rm Im}(\phi(\xi))>0$ for $\phi \in \Phi$.

\subsection{Principal polarizations on CM-abelian varieties} \label{Spptori}

In \cite[page 310]{vanwamelen},
Van Wamelen developed an algorithm to 
produce isomorphism classes of 
principally polarized abelian varieties of type $({\mathcal O}_L,\Phi)$ based on the following result.

\begin{theorem} \label{Tvw345} (Van Wamelen) \cite{vanwamelen}
Let $(L,\Phi)$ be a CM-type.
\begin{enumerate}

\item (Theorem 4)\footnote{We made a small adjustment to the notation to be consistent with the other parts of this theorem.}
Writing $L=L_0(\sqrt{-c})$ for some $c \in {\mathcal O}_{L_0}$, then there exist a fractional ideal $\mathfrak{a} \subset L$ 
and an element $b \in {\mathcal O}_{L_0}$ such that 
$D_{L/\QQ} \cdot \mathfrak{a} \bar{\mathfrak{a}} = \langle b\sqrt{-c}\rangle$.

\item (Theorem 3) Let $\xi\in L$ be such that $L=L_0(\xi)$, $\xi^2 \in L_0$, and 
$D_{L/\QQ}\cdot \mathfrak{a} \bar{\mathfrak{a}} = \langle \xi^{-1}\rangle$,
for some fractional ideal $\mathfrak{a}$ of $L$;
(for example, with the notation of part (1), $\xi=(b\sqrt{-c})^{-1}$).
Define a Riemann form $\EE: \CC^n \times \CC^n \to \CC$ by
\[{\mathbb E}(z,w) = \sum_{i=1}^n \phi_i(\xi)(\bar{z}_i w_i - z_i \bar{w}_i),\]
for $z,w \in \CC^n$.
If  ${\rm Im}(\phi(\xi))>0$ for $\phi \in \Phi$, 
then ${\mathbb E}$ defines a principal polarization on $\CC^n/\Phi(\mathfrak{a})$ of type 
$({\mathcal O}_L, \Phi)$.
Furthermore, if $(L, \Phi)$ is a simple CM-type, then all principal polarizations on $\CC^n/\Phi(\mathfrak{a})$ of 
type $({\mathcal O}_L, \Phi)$ are given by such a $\xi$.\footnote{Note that type $(L, \Phi)$ and type 
$({\mathcal O}_L, \Phi)$ are equivalent in the last two statements.}


\item (Corollary 1) Two principal polarizations of the same CM-type on 
$\CC^n/\Phi(\mathfrak{a})$ arising from $\xi_1$ and $\xi_2$ give isomorphic 
principally {polarized} abelian varieties if and only if there exists a unit $u \in {\mathcal O}_L^*$ such that $\xi_1 = u \bar{u} \xi_2$.
\end{enumerate}
\end{theorem}

\begin{corollary} \label{beta} 
Let $\Phi$ be a CM-type of $L$.
An element $\xi = \beta^{-1}$ for some $\beta \in {\mathcal O}_L$
defines a principal polarization on $A_\Phi=\CC^{n}
/\Phi(\mathcal{O}_L)$ of CM-type $({\mathcal O}_L, \Phi)$
if and only if
\begin{enumerate} \item $\beta$ generates the different $D_{L/\QQ}$; 
\item $\beta=-\overline{\beta}$; 
\item  ${\rm Im}(\phi(\beta))<0$, for each $\phi\in\Phi$.
\end{enumerate}

Two elements $\beta,\beta'$ satisfying the above conditions yield isomorphic principally polarized abelian varieties if and only if there exists a unit $u\in\mathcal{U}_L$ such that $\beta=u \bar{u}\beta'$.
If the CM-type $\Phi$ is simple, 
then all principal polarizations of $A_\Phi$ of CM-type $({\mathcal O}_L,\Phi)$ arise this way. 
\end{corollary}

\begin{proof}
The result follows from Theorem~\ref{Tvw345}, 
replacing $\xi$, $\xi_1$, $\xi_2$ with $\beta^{-1}$, $\beta^{-1}$, $\beta'^{-1}$.
\end{proof}

In Lemma~\ref{beta_lemma}, we determined an element $\beta_0\in\mathcal{O}_F$ satisfying conditions (1) and (2) in Corollary~\ref{beta} when $F=\QQ(\zeta_m)$ for many values of $m$.

\subsection{Existence and uniqueness of principal polarizations}\label{part2}

In this section, we study principal polarizations on CM-abelian varieties.
Under conditions on the class group and unit group of the field, 
we show such principal polarizations exist and are uniquely determined.
Recall that $L$ is a CM-field with maximal totally real subfield $L_0$ and $n=[L_0:\QQ]$.

\begin{proposition}\label{CMuniqueF}
Suppose $L_0$ has units of independent signs and $(L, \Phi)$ is a CM-type.
\begin{enumerate}
\item Then there exists a principally polarized CM-abelian variety of type $({\mathcal O}_L,\Phi)$.
\item If $(L,\Phi)$ is simple, then any CM-abelian variety of type $({\mathcal O}_L,\Phi)$ has at most one principal polarization, up to isomorphism.
\item If $(L, \Phi)$ is simple and $L$ has class number $1$,
then there exists a unique principally polarized CM-abelian variety of type $({\mathcal O}_L,\Phi)$, up to isomorphism.
\end{enumerate}
\end{proposition}

\begin{proof}
By Theorem~\ref{Tvw345}(1), there exists a fractional ideal $\mathfrak{a}$ in $L$ and an element $\beta_0\in\mathcal{O}_L$ satisfying the conditions: 
(1) $\beta_0$ generates $D_{L/\QQ}\mathfrak{a}\overline{\mathfrak{a}}$
and (2) $\beta_0=-\overline{\beta}_0$.
For $\beta' \in \mathcal{O}_L$, then $\beta'$ also satisfies conditions (1) and (2) if and only if $\beta'=u_0 \beta_0$ 
for a totally real unit $u_0 \in \mathcal{U}_{L_0}$.

By Theorem~\ref{Tvw345}(2), $\beta \in \mathcal{O}_L$ defines a principal polarization of type  $({\mathcal O}_L,\Phi)$ on $\CC^n/\Phi(\mathfrak{a})$ if and only if it satisfies conditions (1) and (2), and also 
\begin{enumerate}
\item[(3)]  ${\rm Im}(\phi(\beta))<0$, for each $\phi\in\Phi$.
\end{enumerate}

Hence, to finish the first claim,
it suffices to check that there exists $u_0\in \mathcal{U}_{L_0}$ such that $\beta=u_0 \beta_0$ satisfies 
condition (3), for any CM-type $\Phi$ of $L$. 
If $L_0$ has units of independent signs, then $\rho_{L_0}$ is surjective;
this shows that the unit $u_0$ described above exists.

By Theorem~\ref{Tvw345}(2), if $(L,\Phi)$ is simple, all principal polarizations of type $\Phi$ on $\CC^n/\Phi(\mathfrak{a})$ arise from some $\beta \in \mathcal{O}_L$ satisfying conditions (1)--(3).
By Theorem~\ref{Tvw345}(3), 
$\beta,\beta' \in \mathcal{O}_L$ satisfying conditions (1)--(3) define isomorphic principally polarized abelian varieties if and only $\beta'=N(u) \beta$ 
 for some unit $u\in \mathcal{U}_{L}$.

By definition, given a fractional ideal $\mathfrak{a}$ of $L$, if $\beta \in \mathcal{O}_L$ satisfying conditions (1)--(3) exists,  then  $\beta' \in \mathcal{O}_L$ 
also satisfies conditions (1)--(3) if and only if $\beta'=u^+ \beta$ 
for some totally positive unit $u^+ \in \mathcal{U}^+_{L_0}$.

Hence, to finish the second claim,
it suffices to check that for any $u^+\in \mathcal{U}^+_{L_0}$ there is $u\in {\mathcal U}_L$ such that $u^+=N(u)$.
Since $L_0$ has units of independent signs, this follows from Lemma~\ref{N=R}.
Finally, if $L$ has class number $1$, then by Theorem~\ref{Tvw345} and \cite[Theorem~5]{vanwamelen}, any CM-abelian variety of type $({\mathcal O}_L,\Phi) $ is isomorphic to $A_\Phi=\CC^{n} /\Phi(\mathcal{O}_L)$. 
\end{proof}

\begin{proposition}\label{CMuniqueF2}
Suppose $L_0$ has units of almost independent signs 
and $Q(L)=2$, where $Q(L)$ is defined in Definition~\ref{Dhasseunitindex}.
Suppose $(L, \Phi)$ is a simple CM-type.
\begin{enumerate}
\item The number of isomorphism classes of principal polarizations on a CM-abelian variety of type $({\mathcal O}_L,\Phi)$ is at most one.
\item Suppose in addition that $L$ has class number $1$.
If there exists a principally polarized CM-abelian variety of type $({\mathcal O}_L,\Phi)$, then it is unique up to isomorphism.
\end{enumerate}
\end{proposition}

\begin{proof}
The arguments in the proof of Proposition~\ref{CMuniqueF} still apply, after observing that $U^+_{L_0}=N({\mathcal U}_L)$
when $L_0$ has units of almost independent signs and $Q(L)= 2$ by Lemma~\ref{N=R}.
\end{proof}

\subsection{CM-types for cyclotomic fields} \label{Spart2cyc}

We consider the case when $L$ is the cyclotomic field $F=\QQ(\zeta_m)$. 
By Section~\ref{Sindepsigns}, the next result is a special case of Propositions~\ref{CMuniqueF} and \ref{CMuniqueF2}.

\begin{corollary}\label{CMuniqueCyclo}
Let $m$ such that $F=\QQ(\zeta_m)$ has class number 1.  Let $\Phi$ be a CM-type of $F$.
\begin{enumerate}
\item If $F_0$ has narrow class number $1$, then there exists a principally polarized CM-abelian variety of type $({\mathcal O}_F,\Phi)$.  Furthermore, if $(F, \Phi)$ is simple, then it is unique up to isomorphism. 
\item Suppose $F_0$ has narrow class number $2$ and $\Phi$ is simple.   
If there exists a principally polarized CM-abelian variety of type $({\mathcal O}_F,\Phi)$, then
it is unique up to isomorphism. 
\end{enumerate}
\end{corollary}

In the rest of this section, we explicitly describe the data of the principal polarization.

\begin{corollary}\label{betauniquecyclo}
Let $m$, $F = \QQ(\zeta_m)$, and $\beta_0$ be as in Lemma~\ref{beta_lemma}.
Let $\Phi$ be a CM-type of $F$.
Consider conditions (1)--(3) in Corollary~\ref{beta}.  
\begin{enumerate}
\item If $\Phi'=i\Phi$ and $\beta$ satisfies conditions (1)-(3) for $\Phi$, then $\beta'=\sigma_{i^{-1}} (\beta)$ satisfies conditions (1)-(3) for $\Phi'$. 
\item Suppose $F_0$ has units of independent signs.
Then, for any CM-type $\Phi$ of $F$,
there exists a unit $u_0 \in \cU_{F_0}$ such that $\beta=u_0\beta_0$ satisfies 
conditions (1)--(3) for $\Phi$.  
Furthermore, 
$\beta$ is the unique element satisfying conditions (1)--(3) for $\Phi$, 
up to multiplication by an element of $N_{F/F_0}(\mathcal{U}_{F})$. 
\end{enumerate}
\end{corollary}

The condition that $F_0$ has units of independent signs is satisfied if the narrow class number of $F_0$ 
is odd \cite[bottom of page 58]{bookCH}.

\begin{proof}
\begin{enumerate}
\item This follows from Lemma~\ref{Lchangeembed} and Lemma~\ref{Lchangeembed2}.
\item By Lemma~\ref{beta_lemma}, $\beta_0$ satisfies conditions (1) and (2)  of Corollary~\ref{beta}.
Since $F_0$ has units of independent signs,
there exists $u_0 \in \mathcal{U}_{F_0}$
such that $\beta=u_0\beta_0$ satisfies condition (3).
Since $u_0 \in  \mathcal{U}_{F_0}$ is a real unit,
$\beta$ also satisfies conditions (1) and (2)  of Corollary~\ref{beta}.
Note that $\beta$ is unique up to multiplication by a totally positive unit of $F_0$.
By Lemma~\ref{N=R}, $\mathcal{U}^+_{F_0} = N(\mathcal{U}_{F})$,
proving the uniqueness statement.
\end{enumerate}
 \end{proof}

See Example~\ref{Ebeta5} (resp.\ \ref{Ebeta7})
for an example of Corollary~\ref{betauniquecyclo} when $m=5$ (resp.\ $m=7$).
As another example, when $m=8$, then $\beta_0=-4i$ from Lemma~\ref{beta_lemma}:
if $\cf = (0,1,0,1)$, set $u_0=-1$ and $\beta = 4i$;
if $\cf = (1,1,0,0)$, set $u_0 = \sqrt{2}-1$ and $\beta = 4(1-\sqrt{2})i$.

\section{Shimura data} \label{Sshdata}

Consider a monodromy datum $\gamma=(m,N,a)$ as in Section~\ref{Sfamily}.
Recall from Section~\ref{Sfamily} that $Z_\gamma$ is the
closure in ${\mathcal A}_g$ of the image of the Hurwitz space of $\mu_m$-cyclic covers of ${\mathbb P}^1$
with monodromy datum $\gamma$. 
A point $P$ in $Z_\gamma$ 
represents the Jacobian of a curve $C_P$ (either smooth or of compact type), 
for which there is an admissible $\mu_m$-cover $\psi_P: C_P \to T$ 
with monodromy datum $\gamma$, where $T$ is a tree of projective lines.
Recall that $S_\gamma$ is the connected component of $\Sh(\gamma)$ containing $Z_\gamma$.
In this section, under certain assumptions on $m$, 
we provide a method to determine the Shimura datum of $S_\gamma$.  

\subsection{Distinguished points} \label{SHurwitz}

In this section, we prove that there exists a \emph{distinguished point} $P$ in the family $Z_\gamma$ 
such that we can compute the Shimura datum of the Jacobian of $C_P$.

\begin{definition} \label{Ddistinguished}
A point $P$ in $Z_\gamma$ is a \emph{distinguished point} if the Jacobian of $C_P$ is 
a principally polarized abelian variety with complex multiplication by a maximal order in a CM-field, or
the direct sum of such together with the product polarization.
\end{definition}

\begin{proposition} \label{Pdegenerate}
Let $m$ be an odd prime.  Let $\gamma=(m,N,a)$ be a monodromy datum. 
Then $Z_\gamma$ has a distinguished point $P$.  More specifically, for $r=N-2$:

\begin{enumerate} 
\item In the family of $\mu_m$-covers with monodromy datum $\gamma$, 
there is a point which represents an admissible $\mu_m$-cover $\psi: C_P \to T$, 
where $T$ is a tree of $r$ projective lines and $C_P$ is a curve of compact type, 
with $r$ irreducible components ${\mathcal C}_1, \ldots, {\mathcal C}_r$, 
each of which is a curve of genus $(m-1)/2$ admitting a
$\mu_m$-cover of ${\mathbb P}^1$ branched at $3$ points.

\item 
The Jacobian of $C_P$ is of the form $A \simeq \bigoplus_{j=1}^r A_j$, 
where each $A_j$ is a principally polarized abelian variety of dimension $(m-1)/2$ having complex 
multiplication by {${\mathcal O}_F=\ZZ[\zeta_m]$}, together with the product polarization.
\end{enumerate}
\end{proposition}

\begin{proof}
The fact that $Z_\gamma$ has a distinguished point $P$ is immediate from part (2), 
which we will show follows from part (1).
\begin{enumerate}
\item 
Let $a=(a(1), \ldots, a(N))$ be the inertia type.
If $N\geq 4$ and $m$ is an odd prime, then $a$ has the following property:
there is a pair $(i,j)$ with $1 \leq i < j \leq N$, such that $a(i) + a(j) \not \equiv 0 \bmod m$.
Without loss of generality, we can suppose $i=1$ and $j=2$. 

When $N \geq 4$, then $Z^0_\gamma$ is affine.  
Consider a family of covers, where the first branch point $b_1$ approaches the second branch point $b_2$. 
When $b_1 = b_2$, the curve becomes singular and its normalization is a $\mu_m$-cover of the join of two projective
lines.  Because of the condition $a(i) + a(j) \not \equiv 0 \bmod m$, 
this is a degeneration of compact type as described in \cite[Remark 5.2]{LMPT2}.
It is the admissible join of two $\mu_m$-covers with these monodromy data:
\begin{eqnarray*} 
\gamma_1 & = & (m, 3, (a(1), a(2), -(a(1) +a(2)))); \\ 
\gamma_2 & = & (m, N-1, (a(1) +a(2), a(3), \ldots, a(N))).
\end{eqnarray*}

By induction on $N$, we see that 
the universal family degenerates completely to an admissible $\mu_m$-cover 
$\psi: {\mathcal C} \to T$, where $T$ is a tree of $r$ projective lines and the restriction of $\psi$ above each component 
of $T$ is branched at $3$ points.
By \eqref{Egenus}, each irreducible component of ${\mathcal C}$ has genus $(m-1)/2$.

\item Choose a labeling $C_1, \ldots, C_r$ of the irreducible components of ${\mathcal C}$.
Let $A = {\rm Jac}(\mathcal C)$ and $A_j={\rm Jac}(C_j)$.  Then $A_j$ is a principally polarized abelian variety 
of dimension $(m-1)/2$. 
By \cite[Section 9.2, Equation 8]{BLR}, $A \simeq \bigoplus_{j=1}^r A_j$, and the principal polarization on $A$ 
decomposes as the product polarization.

Consider the CM-field $F=\QQ(\zeta_m)$. 
Then ${\mathcal O}_F = \ZZ[\zeta_m] \subset {\rm End}(A_j)$ since $\mu_m$ acts on $C_j$ for $1 \leq j \leq r$.
Since ${\rm deg}(F/\QQ) = 2 \cdot {\rm dim}(A_j)$,
the abelian variety $A_j$ has complex multiplication by $\mathcal{O}_F$ for $1 \leq j \leq r$.
\end{enumerate}
\end{proof}

\subsection{Shimura datum for Hurwitz spaces} \label{part4}

Let $P \in Z_\gamma$ be a distinguished point as described in Proposition~\ref{Pdegenerate}. 
When $F=\QQ(\zeta_m)$ has class number $1$, 
we determine the Shimura datum for $S_\gamma$.
Then strategy is to explicitly compute
the principally polarized abelian variety ${\rm Jac}(C_P)$ represented by $P$
and then the $F$-vector space $V=H^1(C_P({\mathbb C}), {\mathbb Q})$, 
the lattice $\Lambda=H^1(C_P({\mathbb C}), {\mathbb Z})$,
and the Hermitian matrix. 

\begin{definition} \label{Dbetaj}
Let $P \in Z_\gamma$ be a distinguished point as described in Proposition~\ref{Pdegenerate}. 
Suppose $\QQ(\zeta_m+\zeta_m^{-1})$ has units of independent signs.
(This is guaranteed by Lemma~\ref{Lgoodm} when $m$ is an odd prime such that 
$F=\QQ(\zeta_m)$ has class number $1$.)

For $1 \leq j \leq r$, define $\beta_j$ and $\xi_j=\beta_j^{-1}$ as follows.
The inertia type of the $\mu_m$-cover $C_j \to \PP$  
determines the signature $\cf_j$ by \eqref{DMeqn}, 
which determines the CM-type $\Phi_j$ as in \eqref{CMtype_sign}.
Let $\beta_0$ be as in Lemma~\ref{beta_lemma}.
By Corollary~\ref{betauniquecyclo},
there exists $u_j \in \cU_{F_0}$ such that $\beta_j:=u_j\beta_0$ satisfies 
conditions (1)--(3) of Corollary~\ref{beta} for $\Phi_j$.  
Furthermore, 
$\beta_j$ is the unique element satisfying conditions (1)--(3) for $\Phi_j$, 
up to multiplication by an element of $N_{F/F_0}(\mathcal{U}_{F})$. 
\end{definition}

\begin{remark}
We note that the CM-type $(F, \Phi_j)$ of $A_j$ may not be simple.
For example, for the Moonen special family $M[17]$ with $m=7$, then $(F, \Phi_2)$ is not simple, see 
Section~\ref{Sm7}.
\end{remark}

\begin{theorem}\label{PELdatumbeta}
Let $m$ be an odd prime.
Suppose $F=\QQ(\zeta_m)$ has class number $1$. 
Let $\gamma=(m,N,a)$ be a monodromy datum with $N \geq 4$.  

Let $P \in Z_\gamma$ be a distinguished point as described in Proposition~\ref{Pdegenerate}. 
Let $r=N-2$.
For $1\leq j\leq r$: let $A_j$ be the abelian variety with CM-type $({\mathcal O}_F, \Phi_j)$ from Proposition~\ref{Pdegenerate};
let $\xi_j=\beta_j^{-1}$ be as in Definition~\ref{Dbetaj}.
Then the integral PEL datum for $S_\gamma$ is given as:
\begin{itemize}
\item the $F$-vector space $V=F^{r}$, together with the standard $\mathcal{O}_F$-lattice $\Lambda=(\mathcal{O}_F)^r\subseteq V$;
\item the Hermitian form $H_B=\langle \cdot \,, \cdot \rangle $ on $V$, which takes integral values on $\Lambda$,
defined by
\[\langle x,y\rangle={\rm tr}_{ F/\QQ}(xB\bar{y}^T) \text{ for } 
B={\rm diag} [\xi_1,\dots \xi_r]\in {\rm GL}_r(F)={\rm GL}(V).\]
\end{itemize}
\end{theorem}

It is straight-forward to compute $B$ from $\gamma$; we provide many examples in Section~\ref{Sexample57}.  

\begin{proof} 
By Lemma~\ref{Lgoodm}, if $m$ is prime and $F=\QQ(\zeta_m)$ has class number $1$, 
then $F_0=\QQ(\zeta_m+\zeta_m^{-1})$ 
has narrow class number $1$ and has units of independent signs.
By hypothesis, for $1 \leq j \leq r$, the abelian variety $A_j$ has dimension $g=(m-1)/2$; 
it has complex multiplication by ${\mathcal O}_F=\ZZ[\zeta_m]$ and has
CM-type $({\mathcal O}_F, \Phi_j)$.
Furthermore, $\beta_j$ defines a principal polarization on $A_j$; the corresponding Hermitian form is given by 
${\mathbb E}(\Phi(x), \Phi(y))={\rm tr}_{F/\QQ} (\xi_j x \bar{y})$ for $x,y \in L$ by \eqref{Ehermdef}.

Since $F=\QQ(\zeta_m)$ has class number $1$, there is a unique complex torus of CM-type 
$({\mathcal O}_F, \Phi_j)$ up to isomorphism. Namely, $A_j \simeq \CC^g/\Phi({\mathcal O}_F)$.
The product polarization on ${\rm Jac}(C_P) \simeq \oplus_{j=1}^r A_j$ defines 
an integral Hermitian space of the prescribed signature on $V=H^1(C_P({\mathbb C}), {\mathbb Q})$.
By \cite[Appendix, Proposition 8]{shimuratranscend}, there is a unique Hermitian form of given signature
on ${\rm Jac}(C_P)$.
Thus the integral PEL datum for ${\rm Jac}(C_p)$ is as given in the statement.
The result follows since the integral PEL datum can be determined at a point of $Z_\gamma$,
see \cite[Sections~2.8 and 2.9, also 3.2]{moonenLinearity}.
\end{proof}

\begin{remark} \label{RgoodM}
The values of $m$ for which Theorem~\ref{PELdatumbeta} applies
are $m=3,5,7,11,13,17,19$. 

For $m=3,5,11,17$, it is not necessary to refer to \cite[Appendix, Proposition 8]{shimuratranscend};
The reason is that $(F, \Phi_j)$ is automatically simple for $1 \leq j \leq r$ by 
Lemmas~\ref{LFermatsimple} and \ref{Lnotsimple}.
Thus by Corollary~\ref{CMuniqueCyclo}(1), 
there is a unique principally polarized CM-abelian variety of type $({\mathcal O}_F,\Phi_j)$ up to isomorphism. 
In Section~\ref{Sanother7}, we give an alternative approach for $m=7$ that relies on 
Corollary~\ref{CMuniqueCyclo}(2) and does not refer to 
\cite[Appendix, Proposition 8]{shimuratranscend}.
\end{remark}

The strategy of Theorem~\ref{PELdatumbeta} sometimes applies when $m$ is composite; see Section~\ref{Scompositem}.

\section{Applications for $m$ prime}  \label{Sexample57}

Theorem~\ref{PELdatumbeta} gives a method to determine the integral PEL datum for $S_\gamma$ for all monodromy data $\gamma=(m,N,a)$ when $m=3,5,7,11,13,17,19$.
In this section, we give examples of this.  Specifically:
\begin{itemize}
\item in Section~\ref{Sall3}, when $m=3$, we explicitly determine the Hermitian form for all $\gamma$;
\item in Sections~\ref{Ssome5}-\ref{Sm7}, when $m=5,7$, 
we give explicit examples for all equivalence classes of $\gamma$ when $N=4$.
\item in Section~\ref{Sanother7}, we follow an alternative approach for $m=7$ which avoids the possible occurrence of non-simple CM-types at the distinguished point. 
\end{itemize}

In particular, 
we determine the integral PEL datum for the 6 Moonen special families 
for which the degree $m$ is an odd prime; these
are denoted by $M[n]$ as in \cite[Table 1]{moonen}.

In each section below, let $F=\QQ(\zeta_m)$.  For $i=1,2$, let
$\Phi_i$ be the CM-type of $F$ determined by the signature $\cf_i$, as in \eqref{CMtype_sign}.
Given the $F$-vector space $V=F^{r}$, together with the standard $\mathcal{O}_F$-lattice $\Lambda=(\mathcal{O}_F)^r\subseteq V$, let 
$H_B=\langle \cdot \,, \cdot \rangle $ denote the Hermitian form on $V$ defined by
\[\langle x,y\rangle={\rm tr}_{ F/\QQ}(xB\bar{y}^T) \text{ for } 
B={\rm diag} [\xi_1,\dots \xi_r]\in {\rm GL}_r(F)={\rm GL}(V).\]

\subsection{Integral PEL datum for all families when $m=3$} \label{Sall3}

We find the integral PEL data for all families with $m=3$.

\begin{example} \label{Euandb3}
When $m=3$, then $\beta_0 = 3/(\zeta_3^2-\zeta_3) = \sqrt{-3}$ by Lemma~\ref{beta_lemma}.

For $\cf_1=(1,0)$, set $u_1 =-1$ and $b_1=u_1\beta_0$ so that $b_1=-\sqrt{-3}$.
Then ${\rm Im}(\sigma_1(b_1)) < 0$.

For $\cf_2= (0,1)$, set $u_2=1$ and $b_2=u_2\beta_0$ so that $b_2=\sqrt{-3}$.
Then ${\rm Im}(\sigma_2(b_2)) < 0$.

The CM-type $\Phi_i$ is simple by Lemma~\ref{LFermatsimple} and $b_i$ 
satisfies Lemma~\ref{betauniquecyclo} with respect to $\Phi_i$.
\end{example}

For $m=3$, there are simple formulas relating the signature $\cf$ and the inertia type $a$.  
If $a$ has $d_1$ entries of $1$ and $d_2$ entries of $2$, then $\cf=(f_1,f_2)$ with 
$f_1 = (2d_1+d_2 -3)/3$ and $f_2 = (d_1+2d_2-3)/3$.
Then $d_1=2f_1-f_2+1$ and $d_2 = 2f_2 -f_1 +1$.
One can check that $0 \leq {\rm max}(f_1,f_2) \leq 2 {\rm min}(f_1,f_2) + 1$.

\begin{corollary} \label{Cshimdat3}
Let $m=3$ and $N \geq 4$.
Let $\gamma = (3,N,a)$ be a monodromy datum with signature $(f_1,f_2)$.
Let $g=f_1+f_2$.
Let $S_\gamma$ be the component of the Shimura variety containing $Z_\gamma$. 

Let $F=\QQ(\zeta_3)$ and $\xi = -1/\sqrt{-3}$.
Then the integral PEL datum of $S_\gamma$
has lattice $(\mathcal{O}_F)^g$ with Hermitian form $H_B$ where 
$B\in {\rm GL}_g(\mathcal{O}_F)$ is diagonal with $f_1$ entries of $\xi$ and 
$f_2$ entries of $-\xi$.
\end{corollary}

\begin{proof}
Consider the admissible $\mu_3$-cover $\varphi: {\mathcal C}_P \to T$ represented by the distinguished point from Proposition~\ref{Pdegenerate}.  Here $T$ is a tree of $r=f_1+f_2$ projective lines.
Above $f_1$ (resp.\ $f_2$) components of $T$, the restriction of $\varphi$ is a $\mu_3$-cover 
branched at $3$ points with signature $\cf_1=(1,0)$ (resp.\ $\cf_2=(0,1)$).
The result then follows from Theorem~\ref{PELdatumbeta}.
\end{proof}

In particular, Corollary~\ref{Cshimdat3} includes the three special families 
M[3], M[6], M[10].

\[\begin{array} {|l|l|l|l|}
\hline 
M & (m,N,a) & B \\ \hline
M[3] & (3, 4, (1,1,2,2)) &  {\rm diag} [\xi,-\xi] \\ \hline
M[6] & (3, 5, (1,1,1,1,2)) & {\rm diag} [\xi, \xi, -\xi] \\ \hline
M[10] & (3, 6, (1,1,1,1,1,1))& {\rm diag} [\xi,\xi,\xi,-\xi] \\ \hline
\end{array}.\]

The cases $M[6]: \xi[1,1,-1]$ and $M[10]: \xi[1,1,1,-1]$ match the table on \cite[page 1]{shimuratranscend}.

\subsection{Integral PEL datum when $m=5$} \label{Ssome5}

We illustrate Theorem~\ref{PELdatumbeta} for several well-chosen examples when $m=5$, 
including all families branched at $N=4$ points up to equivalence and one family branched at $N=5$ points.
This includes the two special Moonen families when $m=5$, namely
$M[11]$ and $M[16]$.
A similar result can be obtained for any monodromy datum $\gamma$ with $N>4$ by an inductive process. 

 \begin{example} \label{Ebeta5} The following table summarizes the data for CM-types when $m=5$.
 \[\begin{array}{|c|c|c|c|}
 \hline
 a & \cf & \Phi & \beta \\ \hline
 (4,3,3) & (0,1,0,1) & \{2,4\} & \beta_1=5/(\zeta_5^3 - \zeta_5^2) \\ \hline
 (3,1,1) & (1,1,0,0) & \{1,2\} & \beta_2 = 5/(\zeta_5-\zeta_5^4) \\ \hline
  (1,2,2) & (1,0,1,0) & \{1,3\} & \beta_3 = -\beta_1 \\ \hline
  (2,4,4) & (0,0,1,1) & \{3,4\} & \beta_4 = - \beta_2 \\ \hline
 \end{array}.\]
 
In the $i$th line of the table, 
the CM-type $\Phi_i$ is simple by Lemma~\ref{LFermatsimple} and $\beta_i$ 
satisfies Corollary~\ref{betauniquecyclo} with respect to $\Phi_i$.
The automorphism $\sigma_3$ permutes the rows via the cycle $(1,2,3,4)$ and its inverse $\sigma_2$ permutes the $4$th column via $\beta_1 \to \beta_2 \to - \beta_1 \to - \beta_2 \to \beta_1$. 
 \end{example}
 
\begin{proof}
When $m=5$, then $\beta_0 = 5/(\zeta_5^3-\zeta_5^2)$ by Lemma~\ref{beta_lemma}.  
For $\cf_1= (0,1,0,1)$, set $u_1=1$ and $\beta_1 = u_1 \beta_0$.
We compute that ${\rm Im}(\sigma_j(\beta_1)) < 0$ for $j=2,4$.
For $\cf_2=(1,1,0,0)$, set $u_2 =  (\zeta^3_5-\zeta_5^2)/(\zeta_5-\zeta_5^4)$ 
and $\beta_2=u_2\beta_0$.
We compute that ${\rm Im}(\sigma_j(\beta_2)) < 0$ for $j=1,2$.
The signature $\cf_3=(1,0,1,0)$ (resp.\ $\cf_4=(0,0,1,1)$) is the complex conjugate of $\cf_1$ 
(resp.\ $\cf_2$), which negates the value of $\beta$.
\end{proof}

\begin{corollary} \label{Cm5}
Let $m=5$ and $F=\QQ(\zeta_5)$. 
Every family of $\mu_5$-covers of ${\mathbb P}^1$ with $N=4$ 
is equivalent to either (i), (ii) or $M[11]$ in the table below.
Recall $\beta_1, \beta_2 \in 
{\mathcal{O}}_F$ from Example~\ref{Ebeta5} and let $\xi_i=\beta_i^{-1}$.
For the monodromy data $\gamma=(5,N,a)$ and $r=N-2$ as below, 
the integral PEL datum of $S_\gamma$
has lattice $(\mathcal{O}_F)^r$ with Hermitian form $H_B$
where $B\in {\rm GL}_r(\mathcal{O}_F)$ is as below.
\[\begin{array} {|l|l|l|l|}
\hline 
M & (5,N,a) & B \\ \hline
(i) & (5,4, (1,1,4,4)) & {\rm diag} [\xi_2,-\xi_2] \\ \hline
(ii) & (5, 4, (1,2,3,4)) &  {\rm diag} [-\xi_1,\xi_1] \\ \hline
M[11] & (5,4, (1,3,3,3)) &  {\rm diag} [\xi_1,\xi_2] \\ \hline
M[16] & (5,5,(2,2,2,2,2)) & {\rm diag} [-\xi_1,-\xi_2,-\xi_1] \\ \hline
\end{array}.\]
\end{corollary}

\begin{proof}
Suppose $m=5$ and $N=4$ and let $a$ be the inertia type of $\gamma$.  
If three of the values of $a$ are the same, the family is equivalent to the one with $a=(1,3,3,3)$, 
which is $M[11]$.  
If two of the values of $a$ are the same, the family is equivalent to (i).
If all values of $a$ are distinct, the family is equivalent to (ii).

By Theorem~\ref{PELdatumbeta}, it suffices to find the CM-type $(F, \Phi_i)$ for the abelian varieties $A_i$ 
in the decomposition of the Jacobian of ${\mathcal C}_P$.
We refer to \cite[Remark~5.2, Lemma~6.4]{LMPT2} for information about the 
admissible degeneration, given in short-hand by: 
(i) $(1,1,3) + (2,4,4)$; (ii) $(1,2,2) + (3,3,4)$; $M[11]$ $(1,3,1) + (4,3,3)$; and
$M[16]$ $(2,2,1) + (4,2,4) + (1,2,2)$.
Using the table in Example~\ref{Ebeta5}, we find the entries of the diagonal of $B$.
\end{proof}

\begin{remark} \label{Requivdata}
The family (ii), with monodromy datum $\gamma=(5,4,(1,2,3,4))$, has a second degeneration of the form
$(1,3,1) + (4,2,4)$, whose Hermitian form has matrix
$B'={\rm diag} [\xi_2,-\xi_2]$. 
In this case, there are down-to-earth ways to explain why the Hermitian forms determined by 
$B'$ and $B={\rm diag} [-\xi_1,\xi_1]$ are isomorphic.

First, by Lemma~\ref{Lchangeembed}, the automorphism $\sigma_2$ takes the inertia types in the first degeneration to 
those in the second by multiplying the entries by $3$.
By Corollary~\ref{betauniquecyclo}(1), the action on the entries of $B$ is via $\sigma_{2^{-1}} = \sigma_3$ and
\[\sigma_3(B)={\rm diag} [\sigma_3(-\xi_1),\sigma_3(\xi_1)] = {\rm diag}[\xi_2, -\xi_2] =B'.\] 
Second, for family (ii), $S_\gamma$ has signature type $(1,1,1,1)$; hence its reflex field is $\QQ$ (which is smaller than $F_0\subset \RR$). The matrices $B$ and $B'$ are conjugate under the action of $\sigma_3\in{\rm Gal}(F_0/\QQ)$ and correspond to the two choices of a $\QQ$-linear embedding  $F_0\hookrightarrow \RR$. 
\end{remark}

\begin{remark} \label{RcompareS}
To compare with Shimura's work, write $w=\zeta_5 + \zeta_5^4$.  Then $w^2+w-1=0$.
So $w = (-1 + \sqrt{5})/2$.  
Then $\xi_2/\xi_1 = - w - 1 = -(1+\sqrt{5})/2$ so
$\xi_1/\xi_2 = (1-\sqrt{5})/2$.

Consider the family $\gamma'=(5, 4, (1,1,1,1,1))$.
A careful look at \cite[Section 5]{shimuratranscend} shows that Shimura replaced $\zeta_5$ by $\zeta_5^3$
in his computation for this family.  This has the effect of switching to the family $M[16]$ with $\gamma=(5, 4, (2,2,2,2,2))$;
indeed, Shimura computes that the signature is $(2,0,3,1)$.
By line 4 of the table in Corollary~\ref{Cm5}, 
the family $\gamma$ has $B=-\xi_1[1, 1, \xi_2/\xi_1] = -\xi_1[1,1,-(1+\sqrt{5})/2]$.
This does not exactly match what is written in line 5 of the table on \cite[page 1]{shimuratranscend},
namely $[1,1, \xi_1/\xi_2]=[1,1,(1-\sqrt{5})/2]$, but it has the
same sign signature and thus yields an isomorphic Hermitian form.


Consider the family $\gamma'=(5, 4, (2,1,1,1))$.
The details for this family are not included in \cite{shimuratranscend} 
but it appears that Shimura replaced $\zeta_5$ by $\zeta_5^3$
in his computation for this family also.  
This has the effect of switching to the family $M[11]$ with $\gamma=(5, 4, (1,3,3,3))$.
By line 3 of the table in Corollary~\ref{Cm5}, the family $\gamma$ has
$B= \xi_2[1, \xi_1/\xi_2] = \xi_2[1, (1-\sqrt{5})/2]$.
This matches what is written in line 4 of the table on \cite[page 1]{shimuratranscend}. 

\end{remark}

\subsection{Integral PEL datum when $m=7$} \label{Sm7}

We illustrate Theorem~\ref{PELdatumbeta} for several well-chosen examples when $m=7$, 
including all families branched at $N=4$ points up to equivalence.
This includes the special Moonen family $M[17]$.
A similar result can be obtained for any monodromy datum $\gamma$ with $N>4$ by an inductive process. 

 \begin{example} \label{Ebeta7} The following table summarizes the cases when $m=7$.
 \[\begin{array}{|c|c|c|c|}
 \hline
 a & \cf & \Phi & \beta \\ \hline
  (1,1,5) & (1,1,1,0,0,0) & \{1,2,3\} & \beta_1 = 7/(\zeta_7 - \zeta_7^6)  \\ \hline
 (3,3,1) & (1,0,1,0,1,0) & \{1,3,5\} & \beta_2 =  7/(\zeta_7^3 - \zeta_7^4) \\ \hline
  (2,2,3) & (1,0,0,1,1,0) & \{1,4,5\} & \beta_3 = 7/(\zeta^2_7 - \zeta_7^5)  \\ \hline
   (6,6,2) & (0,0,0,1,1,1) & \{4,5,6\} & \beta_4 = - \beta_1 \\ \hline
(4,4,6)  & (0,1,0,1,0,1) & \{2,4,6\} & \beta_5= -\beta_2 \\ \hline
 (5,5,4) & (0,1,1,0,0,1) & \{2,3,6\} & \beta_6 = -\beta_3 \\ \hline

 (1,2,4) & (1,1,0,1,0,0) & \{1, 2, 4\} & \beta=-\frac{7(\zeta^3_7 - \zeta_7^4) }{(\zeta_7 - \zeta_7^6) (\zeta^2_7 - \zeta_7^5) }  \\ \hline

 (3,1,5) & (0,0,1,0,1,1) & \{3, 5,6\} & \beta'=-\beta  \\ \hline
 \end{array}.\]
\end{example} 

\begin{lemma}
In the $i$th line of the table in Example~\ref{Ebeta7}, for $1\leq i\leq 6$,
the CM-type $\Phi_i$ is simple, and the element $\beta_i$ 
satisfies Corollary~\ref{betauniquecyclo} with respect to $\Phi_i$.
The generator $\sigma_3$ of the Galois group ${\rm Gal}(\QQ(\zeta_7)/\QQ)$ permutes these via $\sigma^{i-1}_3(\beta_1)= \beta_{i}$, for $i=1,\dots, 6$.

In the last two lines of the table,  the element $\beta$ (resp. $\beta'$)
satisfies Corollary~\ref{betauniquecyclo} with respect to the CM-type $\Phi=\{1,2,4\}$ (resp. $\Phi'=\{3,5,6\}$). The CM types $\Phi$,  $\Phi'$ are not simple. 
\end{lemma}
 
\begin{proof}
When $m=7$, then $\beta_0 = 7/(\zeta_7^4-\zeta_7^3)$ by Lemma~\ref{beta_lemma}.  

For $a_1=(1,1,5)$ and $\cf_1= (1,1,1,0,0,0)$, set $u_1=(\zeta_7^4-\zeta_7^3)/(\zeta_7-\zeta_7^6)$ and $\beta_1 = u_1 \beta_0$.
We compute that ${\rm Im}(\sigma_j(\beta_1)) < 0$ for $j=1,2,3$.

Let $\sigma=\sigma_5$, which is a generator of ${\rm Gal}(\QQ(\zeta_7)/\QQ)$ and has inverse $\sigma_3$.
Then, by Lemmas~\ref{Lchangeembed} and \ref{Lchangeembed2}, for $1\leq i\leq 6$, 
the action of $\sigma_5^{i-1}$ changes the inertia type to $a_{i}:=(3^{i-1}) \cdot a_1$ and 
the CM-type to $\Phi_{i}:= 5^{i-1} \cdot \Phi$, which also determines the signature $\cf_i$. 
Also, the element $\beta_{i}=\sigma_3^{i-1}(\beta_1)$ 
satisfies Corollary~\ref{betauniquecyclo} with respect to $\Phi_i$.

For $a=(1,2,4)$ and $\cf=(1,1,0,1,0,0)$, set 
\[u=-\frac{(\zeta_7^4-\zeta_7^3)^2}{(\zeta_7-\zeta_7^6)(\zeta^2_7-\zeta_7^5)}, \ {\rm and} \ 
\beta =u \beta_0 = -\frac{7(\zeta_7^3-\zeta_7^4)}{(\zeta_7 -\zeta_7^6)(\zeta_7^2-\zeta_7^5)}.\]
Then ${\rm Im}(\sigma_j(\beta_2)) < 0$ for $j=1,2,4$.
The last line follows from Lemma~\ref{Lchangeembed} by applying complex conjugation to the previous line. 

By Lemma~\ref{Lnotsimple}, $\Phi_i$ is simple unless it is $\{1,2,4\}$ or $\{3,5,6\}$.
 \end{proof}

For $m=7$ and $N=4$, every family is equivalent to one in the next result.

\begin{corollary} \label{Cm7}
Let $m=7$ and $F=\QQ(\zeta_7)$. 
Recall $\beta_1, \beta_2,\beta_3, \beta \in 
{\mathcal{O}}_F$ from Example~\ref{Ebeta7}. 
Let $\xi_i=\beta_i^{-1}$ for $i=1,2,3$ and $\xi=\beta^{-1}$.
For the monodromy data $\gamma=(7,4,a)$ as below, 
the integral PEL datum of $S_\gamma$
has lattice $(\mathcal{O}_F)^2$ with Hermitian form $H_B$
where $B\in {\rm GL}_2(\mathcal{O}_F)$ is as below.
\[\begin{array} {|l|l|l|l|}
\hline 
M & a & B \\ \hline
(i) & (1,1,2,3) & {\rm diag} [\xi_1,\xi_3] \\ \hline
(ii) &  (1,1,6,6) &  {\rm diag} [\xi_1,-\xi_1] \\ \hline
(iii)& (1,2,5,6) & {\rm diag} [\xi_1,-\xi_1] \\ \hline
M[17] &  (2,4,4,4) &  {\rm diag} [- \xi_2, \xi]  \\ \hline
\end{array}.\]
\end{corollary}

\begin{proof}
We refer to \cite[Remark~5.2, Lemma~6.4]{LMPT2} for information about the 
admissible degeneration of each family, given in short-hand by:
(i) $(1,1,5) + (2,2,3)$; (ii) $(1,1,5) + (2,6,6)$; (iii) $(1,5,1) + (6,2,6)$;
and $M[17]$ $(4,4,6) + (1,4,2)$. 
For each family, 
consider the CM-types $(F, \Phi_1)$ and $(F, \Phi_2)$ for the 
abelian varieties $A_i$ 
in the decomposition of the Jacobian at the distinguished point arising from the degeneration. 
We compute the associated integral PEL data by Theorem~\ref{PELdatumbeta}, using the table in Example~\ref{Ebeta7}.
\end{proof}

\begin{remark}
In cases (i), (ii), and (iii), 
both $(F, \Phi_1)$ and $(F, \Phi_2)$ are simple.
For $M[17]$, the CM-type $(F, \Phi_2)$ is not simple, but this is not a concern by 
\cite[Appendix, Proposition 8]{shimuratranscend} or Section~\ref{Sanother7}.
For $M[17]$, we note that $B = \xi[1, - \xi_2/\xi]$ where $-\xi_2/\xi = \zeta_7+\zeta_7^6-1$.
In line 6 of the table on \cite[page 1]{shimuratranscend}, for the related family $\gamma=(7,4,(4,1,1,1))$,
Shimura writes the Hermitian form as $H_B$ with $B=\xi [1, - {\rm sin}(3\pi/7)/{\rm sin}(2\pi/7)]$. 
\end{remark}

\subsection{Another approach for $M[17]$} \label{Sanother7}

In this section, we compute the integral PEL datum for the family $M[17]$ using another approach.
This approach utilizes a different kind of distinguished point in the family, namely one that represents a 
curve with extra automorphisms.  In this section, we see that the Jacobian of this curve has complex multiplication by a larger field and that its CM-type is simple.

In fact, we proved that every positive-dimensional family of $\mu_7$-covers of ${\mathbb P}^1$ has a distinguished point representing a product of principally polarized abelian varieties, each of which has complex multiplication with a CM-type that is simple. In this way, one can avoid the use of \cite[Appendix, Proposition 8]{shimuratranscend} when $m=7$.  We omit the details of this.

We start by finding another distinguished point in the family $Z_\gamma$, 
under certain restrictive conditions on $\gamma$. 
Similarly to Notation~\ref{Ndefsigma}, for $0 \leq n \leq 3m$ with ${\rm gcd}(n,3m)=1$, 
let $\sigma_n$ be the embedding 
$\QQ(\zeta_{3m}) \hookrightarrow \CC$
determined by $\sigma_n(\zeta_{3m})=\zeta_{3m}^n$.

\begin{proposition} \label{Pdis3}
Let $\gamma = (m,4,a)$ where $m > 3$ is prime and $a = (1,1,1, m-3)$.
Then $Z_\gamma$ has a distinguished point $Q$.
More specifically:
in the family of $\mu_m$-covers with monodromy datum $\gamma$, 
there is a point which represents a $\mu_m$-cover $\psi: C_Q \to {\mathbb P}^1$, 
where $C_Q$ is a curve of genus $m-1$ having an automorphism of order $3$.

Furthermore, the Jacobian of $C_Q$ has complex multiplication by $(\ZZ[\zeta_{3m}],\Phi_Q)$,
where, for any $0<n<3m$ with $(n,3m)=1$, the embedding $\sigma_n$ is in $\Phi_Q$ if and only if 
\begin{align*}  n\in[0, 2m/3]\cup[m,5m/3]\cup [2m, 8m/3]&\text{ if } n\equiv 1\bmod 3, or \\
n\in[0, m/3]\cup[m,4m/3]\cup [2m, 7m/3] &\text{ if } n\equiv 2\bmod 3.  \end{align*} 
\end{proposition}

It is possible to generalize Proposition~\ref{Pdis3}, by replacing $3$ by an odd prime $\ell$ relatively prime to $m$ and letting $N=\ell+1$ and $a=(1,\ldots, 1, m-\ell)$.  We omit this generalization.  

\begin{proof} [Proof of Proposition~\ref{Pdis3}]
Let $C_Q$ be the smooth projective curve with equation $y^m=x^3-1$.
It admits a $\mu_m$-cover to ${\mathbb P}^1$ branched at $1, \zeta_3, \zeta_3^2, \infty$ with inertia type $a=(1,1,1,m-3)$.
The genus of $C_Q$ is $m-1$ by \eqref{Egenus}.
Also $C_Q$ has an automorphism $(x,y) \mapsto (\zeta_3 x, y)$ of order $3$.

The Jacobian $J_Q = {\rm Jac}(C_Q)$ is a principally polarized abelian variety having dimension $m-1$.  Let $F=\QQ(\zeta_m)$, and  $L=\QQ(\zeta_{3m})$. 
The field $L$ is a CM-field of degree $2\cdot\dim (J_Q)$. 
Then, the inclusion $\mathcal{O}_F\subset {\rm End}(J_Q)$ extends to an inclusion $\mathcal{O}_L=\mathcal{O}_F[\zeta_3]\subset {\rm End}(J_Q)$. Thus $J_Q$ has complex multiplication by 
$\mathcal{O}_L=\ZZ[\zeta_{3m}]$.

Consider the morphism $\phi:C_Q\to {\mathbb P}^1$, taking $(x,y)\mapsto x^3$. 
Then $\phi$ is a $\mu_{3m}$-cover of ${\mathbb P}^1$, branched at $0,1,\infty$.
We compute the CM-type of $J_Q$ by finding first the inertia type and then the signature type of $\phi$.
Let $b_0$ (resp.\ $b_1$, resp.\ $b_\infty$) denote the element of $\ZZ/3m\ZZ$ that determines the 
canonical generator of inertia of $\phi$ above $0$, (resp.\ $1$, resp.\ $\infty$). 

First, note that $b_1=3$.  This is because $\phi$ 
is branched at the $3$ points $x=1, \zeta_3, \zeta_3^2$ that lie above $x^3=1$; the canonical generator of 
inertia of $\phi$ at the points of $C_Q$ above these is, by definition, the automorphism identified with
$\zeta_7^1=\zeta_{21}^3$.

Second, without loss of generality, $b_0=m$. 
To see this, note there are $m$ points of $C_Q$ above $x^3=0$, so ${\rm gcd}(b_0, 3m)=m$;
this implies that $b_0 = 2m$ or $b_0=m$; 
possibly after replacing the order $3$ automorphism with its inverse, we can suppose that $b_0 = m$.  
Third, $b_\infty = 3m - b_0 - b_1 = 2m-3$, so $(b_0, b_1, b_\infty)=(m, 3, 2m-3)$. 

Next, we compute the signature $\cf$ of $\phi$ from its inertia type.
By \eqref{DMeqn}, if $0<n<3m$, and $(n,3m)=1$, then
\begin{equation*}
\cf(\sigma_n) = -1  + \langle -n/3 \rangle + \langle -n/m \rangle+\langle (-2n)/3+(n/m) \rangle .\end{equation*}
If  $n\equiv 1 \bmod 3$, then
\begin{align*}
\cf(\sigma_n) = &-1  +  (2/3)   +  \langle -n/m \rangle+\langle (-2/3)+(n/m) \rangle ;\end{align*}
if $n\equiv 2 \bmod 3$, then
\begin{align*}
\cf(\sigma_n) = &-1  +  (1/3)  +  \langle -n/m\rangle +  \langle (-1/3)+(n/m) \rangle.\end{align*}
We deduce that $\cf(\sigma_n)=1$ if $n\equiv 1\bmod 3$ and $n\in[0, 2m/3]\cup[m,5m/3]\cup [2m, 8m/3]$ or if $n\equiv 2\bmod 3$ and $n\in[0, m/3]\cup[m,4m/3]\cup [2m, 7m/3]$; otherwise $\cf(\sigma_n)=0$.

Finally, by (\ref{EdefCMtype}), the CM-type $\Phi_Q$ of $J_Q$ is as given in the statement. 
\end{proof}

For $m=7$, we verify that the CM-type $\Phi_Q$ of the Jacobian of $C_Q$ is simple.

\begin{lemma} \label{Lm7simple}
When $m=7$, then the CM-type $\Phi_Q$
in Proposition~\ref{Pdis3} is simple.
\end{lemma}

\begin{proof}
Let $C_Q$ be the smooth projective curve with equation $y^7=x^3-1$. 
From Proposition~\ref{Pdis3}, 
its CM-type is $\Phi = \{\tau_n \mid n=1,2,4,8,10,16\}$.
We check that $\Phi$ is simple by showing that it is not induced from any proper CM-field.
Let $\sigma_i: \QQ(\zeta_{3m}) \to \CC$ be the embedding given by $\sigma_i(\zeta_{3m}) = \zeta_{3m}^i$, for 
$1 \leq i \leq 3m$ with ${\rm gcd}(i,3m)=1$.  
Every subgroup of ${\rm Gal}(\QQ(\zeta_{3m})/\QQ)$ not containing complex conjugation 
contains either $\sigma_8$ (order $2$), $\sigma_{13}$ (order $2$) or $\sigma_4$ (order 3).
It suffices to show that the subgroup generated by each of these has a coset not 
contained in $\Phi$ but having non-trivial intersection with $\Phi$, so that $\Phi$ is not a union of cosets.
For $\langle \sigma_8 \rangle$, this is true because of the coset $\{\sigma_4, \sigma_{11}\}$.
For $\langle \sigma_{13} \rangle$, this is true because of the coset $\{\sigma_1,\sigma_{13}\}$.
For $\langle \sigma_{4} \rangle$, this is true because of the coset $\{\sigma_2, \sigma_8, \sigma_{11}\}$.
\end{proof}

The following lemma will be helpful.

\begin{lemma} \label{Lcompute7a}
In the field $L=\QQ(\zeta_{21})$, write $\zeta=\zeta_{21}$.
Let $\beta_3 =  7/(\zeta^6-\zeta^{15})$.
Let 
\begin{equation}\label{Edefalpha}
\alpha = (\zeta^7-\zeta^{14})(\zeta^2-\zeta^{19})= (\zeta^9+\zeta^{12}) - (\zeta^5 + \zeta^{16}).
\end{equation}
With respect to the CM-type $\Phi=\{1,2,4,8,10,16\}$ of $L$, the following element $z$  
satisfies conditions (1)--(3) in Corollary ~\ref{beta}:
\begin{equation} \label{Edefz}
z= \beta_3 \alpha = 21(\zeta^{2}-\zeta^{19})/((\zeta^{14}-\zeta^{7})(\zeta^{6}-\zeta^{15})).
\end{equation}
\end{lemma}

\begin{proof}
We verify the first claim by computation using that $\zeta^{14}-\zeta^7 = - \sqrt{-3}$. 
We verify the second claim from Lemma~\ref{beta_lemma}, part (4).
Note that $z=-\sigma_4(\beta_0)$ for $\beta_0$ as in {\em loc. cit}.
\end{proof}

We now compute the integral PEL datum of $M[17]$ using Proposition~\ref{Pdis3}.  

\begin{proposition}\label{M17}
Let $m=7$ and $F=\QQ(\zeta_7)$. 
Let $\xi_1=\beta_1^{-1}$ where $\beta_1  \in 
{\mathcal{O}}_F$ is defined in Example~\ref{Ebeta7}. Consider the units $u_1=\zeta_7^2+\zeta_7^5$ 
and $v=(1+\zeta_7^3+\zeta_7^4)^{-1}$
of $\ZZ[\zeta_7]$.
For the family $M[17]$, with monodromy datum $\gamma=(7,4,(2,4,4,4))$, 
the integral PEL datum of $S_\gamma$
has lattice $\Lambda=(\mathcal{O}_F)^2$ with Hermitian form $H_B$ 
where \[B= \xi_1v \begin{bmatrix} 
u_1& - \zeta_7^2 \\ - \zeta_7^5 &  u_1
\end{bmatrix}.\]
\end{proposition}

\begin{proof}
Consider the family $Z=Z_\gamma$ with monodromy datum $\gamma=(7,4,(1,1,1,4))$. 
Because $(2,4,4,4)=4 \cdot (4,1,1,1)$ and $4\equiv 2^{-1} \bmod 7$, the action of $\sigma_2$ takes $Z$ to the family $M[17]$.  

Let 
\begin{equation} \label{EdefA}
A=  \frac{\zeta_7^2-\zeta_7^5}{7(1+\zeta_7+\zeta_7^6)} \begin{bmatrix} 
\zeta_7^3+\zeta_7^4& -\zeta_7^4 \\ - \zeta_7^3 & \zeta_7^3+\zeta_7^4
\end{bmatrix} . 
\end{equation}
One can check that $A$ is integral away from $7$.
We claim that the integral PEL datum of $Z$ is given by the lattice $\Lambda=(\mathcal{O}_F)^2$ and the Hermitian form for the matrix $A$.  This is sufficient to prove the proposition because,  
by the last statement of Corollary~\ref{beta}, 
the integral PEL datum of $M[17]$ is then given by the lattice 
$\Lambda=(\mathcal{O}_F)^2$ and the Hermitian form for the matrix 
$B=\sigma_2^{-1}(A)$.

We turn to computing the integral PEL datum of $Z$. 
Let $L=\QQ(\zeta_{21})$.  Consider the CM-type $\Phi=\{1,2,4,8,10,16\}$ for $L$.
Consider the curve $C_Q$ of genus $6$ given by the equation $y^7=x^3-1$ and the associated 
$\mu_7$-cover $C_Q \to {\mathbb P}^1$.
By Proposition~\ref{Pdis3}: this cover is represented by a point $Q$ of $Z$; 
the Jacobian ${\rm Jac}(C_Q)$ is a principally polarized abelian variety with 
complex multiplication by $({\mathcal O}_L,\Phi_Q)$; and this CM-type is simple by Lemma~\ref{Lm7simple}.
In other words, $Q$ is a simple distinguished point of $Z$.  

Note that $L$ has class number 1 and $L_0$ has narrow class number 2. 
By Corollary~\ref{CMuniqueCyclo}, there exists a unique
principally polarized CM-abelian variety of type $({\mathcal O}_L,\Phi)$ up to isomorphism. 
By Lemma~\ref{Lcompute7a}, the element $z$ from \eqref{Edefz} 
satisfies conditions (1)--(3) in Corollary ~\ref{beta} with respect to $\Phi$.   
Thus ${\rm Jac}(C_Q)$ is isomorphic to the torus $\CC^6/\Phi(\mathcal{O}_L)$, 
with principal polarization given by $\xi=z^{-1}$. 
Furthermore, the integral PEL datum of $Z$ is given by the lattice 
$\Lambda=\mathcal{O}_L\subset V=L$ and the Hermitian form 
$\langle x,y\rangle={\rm tr}_{L/Q}(x\xi \bar{y}^T)$.
Note that $\mathcal{O}_L=\ZZ[\zeta_{21}] = \mathcal{O}_F[\zeta_3]$.
With respect to the ordered basis $1, \zeta_3$ for $\mathcal{O}_L$ over $\mathcal{O}_F$, 
the Hermitian form is given by a matrix in ${\rm GL}_2(\mathcal{O}_F)$.
By Lemma~\ref{Lcompute7b}, this matrix is $A$.
\end{proof}

\begin{lemma} \label{Lcompute7b}
With respect to the ordered basis $1, \zeta_3$ for $\mathcal{O}_L$ over $\mathcal{O}_F$, 
the Hermitian form 
$\langle x,y\rangle={\rm tr}_{L/Q}(x\xi \bar{y}^T)$
is given by the matrix $A \in {\rm GL}_2(\mathcal{O}_F)$, where $A$ is given in \eqref{EdefA}.
\end{lemma}

\begin{proof}
Let $x,y \in \Lambda$.  By Lemma~\ref{Lcompute7a}, since $\beta_3 = 7/(\zeta_7^2-\zeta_7^5)$ is in 
${\mathcal O}_{F_0}$, we can write: 
\[\langle x,y\rangle={\rm tr}_{L/\QQ}(x z \bar{y}) = {\rm tr}_{F/\QQ}(\beta_3^{-1} {\rm tr}_{L/F}(x \alpha^{-1} \bar{y}))\]
where $\alpha$ is as in Equation \eqref{Edefalpha}.

Let $\tau=\sigma_8 \in {\rm Gal}(L/\QQ)$.  Then $\tau$ is the generator of ${\rm Gal}(L/F)$. 
Then
\[ {\rm tr}_{L/\QQ}(x \alpha^{-1} \bar{y})= {\rm tr}_{F/\QQ}(x \alpha^{-1} \bar{y}+\tau(x)\tau(\alpha^{-1})\tau(\bar{y})).\] 
Write $x=x_1+x_2\zeta_3$ and $y=y_1 + y_2 \zeta_3$, for $x_1,x_2, y_1,y_2 \in\mathcal{O}_{F}$.
We compute 
\[
{\rm tr}_{L/\QQ}(x \alpha^{-1} \bar{y}) =  
{\rm tr}_{F/\QQ}((x_1+x_2\zeta_3) \alpha^{-1} (\bar{y}_1+\bar{y}_2\zeta_3^2) + 
(x_1+x_2 \zeta_3^2) \tau(\alpha^{-1})(\bar{y}_1+\bar{y}_2\zeta_3)).\]
Write $a_{1,1}=\alpha^{-1}+\tau(\alpha^{-1})$, $a_{1,2} = \zeta_3^2 \alpha^{-1}+ \zeta_3 \tau(\alpha^{-1})$, and 
$a_{2,1} = \zeta_3 \alpha^{-1}+ \zeta_3^2 \tau(\alpha^{-1})$.  Then
\[{\rm tr}_{L/\QQ}(x \alpha^{-1} \bar{y}) =  
 {\rm tr}_{F/\QQ}(x_1\bar{y}_1 a_{1,1} 
+ x_1\bar{y}_2 a_{1,2} 
+x_2\bar{y}_1 a_{2,1} 
+  x_2\bar{y}_2 a_{1,1}).\]
We compute that 
$a_{1,1} = \alpha^{-1} + \tau(\alpha^{-1}) = (\zeta^4_7+\zeta_7^3)/(1+\zeta_7+\zeta_7^6)$,
$a_{1,2} = - \zeta_7^4/(1+\zeta_7+\zeta_7^6)$, and $a_{2,1} = -\zeta_7^3/(1+\zeta_7+\zeta_7^6)$.
Thus, 
$\langle x,y\rangle
=  {\rm tr}_{F/\QQ}((x_1,x_2)A(\bar{y}_1,\bar{y}_2)^T)$.
\end{proof}

\section{Applications for composite $m$} \label{Scompositem}

We illustrate how to extend Theorem~\ref{PELdatumbeta}
by computing the integral PEL datum in some cases when $m$ is composite.
In particular, we determine the integral PEL datum for $6$ of the Moonen special families
where $m=4,6,10$.

\begin{remark} \label{Rmoonenomit}
Our techniques do not apply well for symplectic Shimura varieties, so
we exclude: the modular curve $M[1]$ with $m=2$;
the Picard surface $M[2]$ with $m=2$; the family $M[7]$ with $m=4$; and $M[12]$ with $m=6$.
We exclude $M[15]$ with $m=8$ and $M[20]$ with $m=12$, 
since the CM-types for bi-quadratic CM-fields are never simple.
Our techniques are not sufficient to handle $M[13]$ with $m=6$ or $M[19]$ with $m=9$.
\end{remark}

\subsection{Difficulties when $m$ is composite}

The key to Theorem~\ref{PELdatumbeta}
is the existence of a distinguished point $P$ in $Z_\gamma$ satisfying the 
properties in Proposition~\ref{Pdegenerate}. 
When $m$ is composite, this is not always possible for the following reasons.

\begin{remark} \label{Rcomposite1}
One issue when $m$ is composite is that the curve $C_P$ might not have compact type.
For example, the Moonen family $M[12]$, with $\gamma=(6,4,(1,1,1,3))$,
has no admissible degenerations of compact type.  The reason is that the two covers 
with inertia types $a_1=(1,1,4)$ and $a_2=(2,1,3)$
would need to be joined at two points, leading to a cycle in the dual graph of 
$C_P$.\footnote{There is a typo in this case in \cite[Lemma 6.4]{LMPT2}.}
The same is true for $M[7]$, with $\gamma = (4,4,(1,1,1,1))$.
\end{remark}

\begin{remark} \label{Rcomposite2}
Another issue when $m$ is composite is 
that the integral group algebra $\ZZ[\mu_m]$ has
non-trivial index in $\prod_{1\leq d|m}\ZZ[\zeta_d]$. 
Hence, the Jacobian of a $\mu_m$-cover of $\bP^1$ branched at 3 points might not have 
complex multiplication by a maximal order. 
See Section~\ref{other_sec2}.
\end{remark} 

\subsection{Integral PEL data for two families with $m=4$} \label{Sm4}

We find the integral PEL datum for two of the three special families with $m=4$: 
$M[4]$ and $M[8]$; we exclude $M[7]$ as it is not unitary.  When $m=4$, the hypotheses of Proposition~\ref{Pdegenerate} are not satisfied
but we can find a distinguished point in the family by direct computation.

\begin{example} \label{Euandb4}
If $m=4$, then $\beta_0 = - 2i$ from Lemma~\ref{beta_lemma}.
Set $u_1=-1$, so $\beta_1=-\beta_0 = 2i$.
If $\cf_1= (0,1)$, then $\beta_1$ satisfies the 3 conditions of Lemma~\ref{beta} 
and all CM-types of $F=\QQ(i)$ are simple. 
\end{example}

\begin{corollary}\label{coro4}
Let $m=4$ and $F=\QQ(i)$. Let $\xi=\beta_1^{-1}= 1/2i$.  For the family $M$ with monodromy datum $\gamma$ below, 
the integral PEL datum for $S_\gamma$
has lattice $(\mathcal{O}_F)^r$ with Hermitian form $H_B$
where $r$ and $B\in {\rm GL}_r(\mathcal{O}_F)$ are:
 \[
 \begin{array}{|l|l|l|l|} 
 \hline
 M& \gamma=(m,N,a) & r & B \in  {\rm GL}_r(\mathcal{O}_F)\\\hline
M[4] & (4, 4, (1,2,2,3)) & 2 & {\rm diag} [\xi, -\xi]\\ \hline
M[8] &  (4, 5, (1,1,2,2,2)) &3 & {\rm diag} [\xi,-\xi,\xi] \\ \hline
\end{array}\]
\end{corollary}

\begin{proof}
When $m=4$, then $a=(3,2,3)$ has signature $\cf=(0,1)$ and $\sigma_2(\beta_1) <0$.
Also $a=(1,2,1)$ has signature $\cf=(1,0)$ and $\sigma_1(-\beta_1) <0$.
By \cite[Lemma 6.4]{LMPT2}, there is a distinguished point in the family $M[4]$ and the family $M[8]$.
The family $M[4]$ has an admissible degeneration, expressed in short hand as $(1,2,1) + (3,2,3)$.
The family $M[8]$ has an admissible degeneration, expressed in short hand as $(1,2,1) + (3,2,3) + (1, 2, 1)$.
The result then follows from Theorem~\ref{PELdatumbeta}.
\end{proof}

The case $M[8]: \xi[1,-1,1]$ matches line 3 of the table on \cite[page 1]{shimuratranscend}

\subsection{Remark when $m=9$} \label{Sm9}

The techniques in this paper are not sufficient to find the integral PEL datum for the Moonen special family $M[19]$ when $m=9$.  This family has inertia type $a=(3,5,5,5)$ and its
only admissible degeneration can be expressed in short hand as $(3,5,1)+(8,5,5)$.
For $a_2=(8,5,5)$, the Jacobian has complex multiplication by $\QQ[\zeta_3]\times \QQ[\zeta_9]$. As the action of $\ZZ[\mu_9]$ does not extend to the maximal order $\ZZ[\zeta_3]\times \ZZ[\zeta_9]$,  we do not know how to compute its lattice.

We expect that our technique is sufficient to determine the integral PEL datum when $m=9$, $N=4$ and $a=(3,1,8,6)$.  We leave the details to the reader to check.

\subsection{Generalities for twice a prime}

\begin{notation}\label{beta610}
Let $m=2m'$ where $m'$ is an odd prime.  
Let $F=\QQ(\zeta_m)$ and $F'=\QQ(\zeta_{m'})$.
We identify $F=F'$ and $\ZZ[\zeta_{m}]=\ZZ[\zeta_{m'}]$, 
by $\zeta_{m}=- \zeta_{m'}^{(m'+1)/2}$.
There is a bijection between CM types $\Phi$ for $F$ and CM types $\Phi'$ for $F'$ where
$\sigma_j \in \Phi$ for an odd integer $1 \leq j \leq m$ if and only if 
$\sigma_{j'} \in \Phi'$, where $j'$ is the reduction of $j$ modulo $m'$.
\end{notation}

If $\beta$ satisfies Corollary~\ref{betauniquecyclo} with respect to the CM type
$(\ZZ[\zeta_{m'}], \Phi')$ then it also satisfies 
Corollary~\ref{betauniquecyclo} with respect to the CM type
$(\ZZ[\zeta_{m}], \Phi)$.
This can be verified in general but we only need a few cases from
Example~\ref{Euandb3} when $m'=3$ and Example~\ref{Ebeta5} when $m'=5$. 

\begin{example} \label{Ebeta610}
\[\begin{array}{|c|c|c|c|}
\hline
m & \beta & \text (\ZZ[\zeta_{m'}], \Phi') & (\ZZ[\zeta_m], \Phi) \\ \hline
6 & b_2=\sqrt{-3} & (\ZZ[\zeta_3],\{2\}) & (\ZZ[\zeta_6], \{5\}) \\ \hline
10 & \beta_1 =  5/(\zeta_5^3-\zeta_5^2) & (\ZZ[\zeta_5],\{2,4\}) & (\ZZ[\zeta_{10}],\{7,9\})  \\ \hline 
10 & \beta_2 =  5/(\zeta_5-\zeta_5^4) &  (\ZZ[\zeta_5], \{1,2 \}) & (\ZZ[\zeta_{10}], \{ 1,7\}) \\ \hline
\end{array}.\]
\end{example}

\subsection{Integral PEL data for 4 families with $m=6,10$} \label{other_sec2}

We find the integral PEL datum for three special families with $m=6$, namely $M[5]$, $M[9]$ and $M[14]$, and for the special family with $m=10$, namely $M[18]$. 
For $m=6$, we exclude $M[12]$ as it is not unitary and $M[13]$ (see Remark~\ref{M13}).
For $m$ composite, the hypotheses of Proposition~\ref{Pdegenerate} are not satisfied.
In each case, we find a distinguished point $P$ in the family
for which the issues in Remarks~\ref{Rcomposite1} and \ref{Rcomposite2} do not occur. 

\begin{corollary}\label{coro10}
Recall Notation~\ref{beta610}.
For the special family $M[18]$ with monodromy datum $\gamma=(10, 4, (3,5,6,6))$,
the integral PEL datum of $S_\gamma$ has lattice
$\Lambda=(\mathcal{O}_{F'})\oplus (\mathcal{O}_F)^2$ with Hermitian form $H_B$
where $B\in {\rm GL}_{1}(\mathcal{O}_{F'})\times {\rm GL}_2(\mathcal{O}_F)$ is
${\rm diag} [\xi_2]\oplus {\rm diag}[\xi_1,\xi_2]$ where $\xi_i=\beta_i^{-1}$.
\end{corollary}

\begin{proof}
Let $m=10$ and $m'=5$.
Let $M=M[18]$ with $\gamma = (10, 4, (3,5,6,6))$ and signature type $\cf=(1,1,0,1,0,0,2,0,1)$.  
The proof is similar to that of Theorem~\ref{PELdatumbeta}.  
We produce a simple distinguished point $P$ in the family similar to that in Proposition~\ref{Pdegenerate}(1);
it represents an
admissible $\mu_m$-cover $\psi : C_P \to T$, where $T$ is a tree of projective lines and $C_P$ is a curve of compact type such that each of its irreducible components is a curve admitting a $\mu_m$-cover of $\bP^1$ branched at $3$ points. 
We verify by direct computation that each irreducible component of $C_P$ has complex multiplication by either $\ZZ[\zeta_m]$ or $\ZZ[\zeta_{m'}]$.

Consider the $\mu_5$-cover $\psi_2 : C_2 \to \bP^1$ branched at three points, with inertia type $a_2=(4, 3, 3)$, and signature $\cf_2=(0,1,0,1)$.  
Then $A_2={\rm Jac}(C_2)$ has complex multiplication by $(\ZZ(\zeta_5), \{2,4\})$.
Consider the induced curve 
$\tilde{C}_2={\rm Ind}_5^{10}(C_2)$, which is the disconnected curve consisting of two copies of $C_2$, and the induced $\mu_{10}$-cover
$\Psi_2: \tilde{C}_2 \to \bP^1$.
Then $\Psi_2$ is branched at three points and, somewhat imprecisely, 
we can say that it has inertia type $(8,6,6)={\rm Ind}_5^{10}(4,3,3)$.
Above the first branch point, there are two points $\eta_2$ and $\eta_2'$ on $\tilde{C}_2$, and they are labeled by the two cosets of $\mu_5 \subset \mu_{10}$.
Let $\mathcal{A}_2 = {\rm Jac}(\tilde{C}_2) \simeq A_2^2$. 

As explained in \cite[Section 3.1]{LMPT3}, the signature type of $\mathcal{A}_2$ is
\[\cf_2=(0,1,0,1,0,0,1,0,1)=(0,1,0,1,0,0,0,0,0)+(0,0,0,0,0,0,1,0,1).\] 
The action of $\ZZ[\mu_{10}]$ on $\mathcal{A}_2$ is the diagonal action of $\ZZ[\zeta_5]\times \ZZ[\zeta_{10}]$ on $A^2$.  The first (resp.\ second) copy of $A_2$ has CM-type 
$(\ZZ[\zeta_5], \{2,4\})$ (resp.\ $(\ZZ[\zeta_{10}], \{7,9\})$);
for this CM-type, the element $\beta_1$ (resp.\ $\beta_1$) satisfies Corollary~\ref{betauniquecyclo}, as seen in Example~\ref{Ebeta610}. 

Consider the $\mu_{10}$-cover $\psi_1 : C_1 \to \bP^1$ branched at three points, with inertia type $a_1=(3, 5, 2)$.
Above the 3rd branch point, there are two points $\eta_1$ and $\eta_1'$ on 
$C_1$ and they are labeled by the two cosets of $\mu_5 \subset \mu_{10}$.
Let $A_1 = {\rm Jac}(C_1)$; 
it has signature type $\cf_1 = (1,0,0,0,0,0,1,0,0)$. By \cite[Lemma 3.1]{LMPT1}, $A_1$ has complex multiplication by $(\ZZ[\zeta_{10}],\{1,7\})$.
Since $\{1, 7\} = \{3 \cdot 7, 3 \cdot 9\} \bmod 10$, the element $\beta$ that satisfies 
Corollary~\ref{betauniquecyclo} is $\sigma_7(\beta_1) = \beta_2$.
All CM types of $\QQ(\zeta_5)=\QQ(\zeta_{10})$ are simple. 

Let $\mathcal{C}$ denote the singular curve, whose components are $C_1$ and $\tilde{C}_2$, formed by identifying $\eta_1$ $\eta_1'$ with $\eta_2$ and $\eta_2'$, so that the cosets are matched correctly,
in two ordinary double points.
The curve $\mathcal{C}$ admits an admissible $\mu_{10}$-cover $\Psi$ to a chain of two projective lines.  By construction, $\mathcal{C}$ has an action by ${\mathcal O}_{F'} \oplus 
{\mathcal O}_{F}^2$, with CM-type given by ${\rm diag} [\xi_2]\oplus {\rm diag}[\xi_1,\xi_2]$. 

The last thing to check is that $\Psi$ is represented by a point $P \in Z_\gamma$.
Since $\Psi$ is admissible, it can be deformed to a $\mu_{10}$-cover of $\bP^1$ with inertia type $(3,5,6,6)$.  This has signature type $\cf=(1,1,0,1,0,0,2,0,1)$. 
Hence, $\Psi$ is represented by a point $P \in Z_\gamma$ and by the preceding paragraph
$P$ is a simple distinguished point.
\end{proof}

\begin{corollary}\label{coro610}
Recall Notation~\ref{beta610} and let $\xi=1/\sqrt{-3}$.
For the special families $M$ with monodromy datum $\gamma=(6,N,a)$ as below, 
the integral PEL datum of $S_\gamma$ has lattice
$\Lambda=(\mathcal{O}_{F'})^{r'}\oplus (\mathcal{O}_F)^r$ with Hermitian form $H_B$ 
where $r',r $, and  $B\in {\rm GL}_{r'}(\mathcal{O}_{F'})\times {\rm GL}_r(\mathcal{O}_F)$ are:
\[
 \begin{array}{|l|l|l|l|} 
  \hline
 M& \gamma=(m,N,a) & (r',r)& B \in {\rm GL}_{r'}(\mathcal{O}_{F'})\times {\rm GL}_r(\mathcal{O}_F)\\\hline
M[5] & (6, 4, (2,3,3,4)) & (0,2) &{\rm diag} [\xi, -\xi]\\ \hline
M[9]  &(6, 4, (1,3,4,4)) &(1,2) & {\rm diag} [-\xi]\oplus{\rm diag}[\xi, -\xi]\\ \hline
M[14] &(6, 5, (2, 2, 2, 3,3)) &(1,3)& {\rm diag} [\xi]\oplus{\rm diag}[\xi,-\xi, \xi] \\ \hline
\end{array}\]
\end{corollary}

\begin{proof}
The proof is very similar to that of Corollary~\ref{coro10} so we provide only a short 
sketch.

\begin{enumerate}
\item Let $M=M[5]$ with $\gamma=(6, 4, (2,3,3,4))$ and $\cf=(1,0,0,0,1)$. 
Then $C_P$ is the join of two $\mu_6$-covers
with inertia types $a_1 = (1,2,3)$ and $a_2 = ( 3, 4, 5)$.
These have signatures $\cf_1=(1,0,0,0,0)$ and $\cf_2=(0,0,0,0,1)$.
 By \cite[Lemma 3.1]{LMPT1}, $A_1$ has CM by $\mathcal{O}_{F}$ of type $\Phi_1=\{1\}$, 
and $A_2$ has CM by $\mathcal{O}_{F}$ of type $\Phi_2=\{5\}$.  
Let $b_1=-b_2 = -\sqrt{-3}$.
For $i=1,2$, the CM-type $\Phi_i$ is simple, and the element $b_i \in \mathcal{O}_{F}$ satisfies Corollary~\ref{betauniquecyclo} with respect to $\Phi_i$.

\item Let $M=M[9]$ with $\gamma=(6, 4, (1,3,4,4))$ and $\cf=(1,1,0,0,1)$. 
Then $C_P$ is the join of two covers
with inertia types $a_1 = (1,2,3)$ and $a_2={\rm Ind}^6_3(2, 2, 2)$.
These have signatures $\cf_1=(1,0,0,0,0)$ and $\cf_2=(0,1,0,0,1)$.
By \cite[Lemma 3.1]{LMPT1}, $A_1={\rm Jac}(C_1)$ has CM-type $(\mathcal{O}_{F}, \{1\})$, which has $\beta= b_1 = -b_2$.
Also ${\mathcal A}_2 \simeq A^2$ and the action of $\ZZ[\mu_{6}]$ on $\mathcal{A}_2$ is given by the diagonal action of $\ZZ[\zeta_3]\times \ZZ[\zeta_{6}]$ on $A^2$.  
The first copy of $A$ has CM-type $(\ZZ[\zeta_3], \{2\})$ which has $\beta=b_2$
and the second copy of $A$ has CM-type $(\ZZ[\zeta_6], \{5\})$ which has $\beta = b_2$.
  
\item Let $M=M[14]$ with $\gamma=(6, 5, (2, 2, 2, 3,3))$ and $\cf=(2,0,0,1,1)$. 
Then $C_P$ is the join of three covers
with inertia types $a_1 = (1,2,3)$,  $a_2 = (3, 4, 5)$ and $a_3={\rm Ind}^6_3(1, 1, 1)$.
These have signatures $\cf_1=(1,0,0,0,0)$, $\cf_2=(0,0,0,0,1)$, and $\cf_3=(1,0,0,1,0)$.
By \cite[Lemma 3.1]{LMPT1}, $A_1$ has CM-type $(\mathcal{O}_{F},\{1\})$ 
which has $\beta=b_1$ and $A_2$ has 
CM-type $(\mathcal{O}_{F}, \{5\})$ which has $\beta=b_2$. 
Also ${\mathcal A}_3 \simeq A^2$ and the action of $\ZZ[\mu_{6}]$ on $\mathcal{A}_3$ is given by the diagonal action of $\ZZ[\zeta_3]\times \ZZ[\zeta_{6}]$ on $A^2$.  
The first copy of $A$ has CM-type $(\ZZ[\zeta_3], \{1\})$ which has $\beta=b_1$
and the second copy of $A$ has CM-type $(\ZZ[\zeta_6], \{4\})$ which has $\beta = b_1$.
\end{enumerate}
\end{proof}

We checked Corollaries~\ref{coro10} and \ref{coro610} independently, using the fact that the Moonen special families for $m=6,10$ are subspaces of those for $m'=3,5$, up to a Galois twist.  

\begin{remark}\label{M13}
 Let $M=M[13]$ be the Moonen special family with $\gamma =(6, 4, (1,1,2,2))$ and $\cf=(2,1,0,1,0)$. 
This has an admissible degeneration to the join of two covers
with inertia types $a_1 = (1,1,4)$ and $a_2={\rm Ind}^6_3(1,1,1)$.
The first one has signature $\cf_1=(1,1,0,0,0)$ and its Jacobian $A_1$ has CM by an order of 
finite index in $\mathcal{O}_{F'}\times \mathcal{O}_{F}$.  
The index is non-trivial since $\langle x^3-1, x^3+1\rangle = \langle 2 \rangle$ in $\ZZ[x]/\langle x^6-1 \rangle$, so
our technique does not apply.
The other degeneration to $a_1 =(1,2,3)$ and $a_2=(3,1,2)$ does not have compact type.
\end{remark}

\bibliographystyle{amsplain}
\bibliography{infbib}

\providecommand{\bysame}{\leavevmode\hbox to3em{\hrulefill}\thinspace}
\providecommand{\MR}{\relax\ifhmode\unskip\space\fi MR }
\providecommand{\MRhref}[2]{%
  \href{http://www.ams.org/mathscinet-getitem?mr=#1}{#2}
}
\providecommand{\href}[2]{#2}
\begin{thebibliography}{10}

\bibitem{BLR}
Siegfried Bosch, Werner L{\"u}tkebohmert, and Michel Raynaud, \emph{N\'eron
  models}, Ergebnisse der Mathematik und ihrer Grenzgebiete (3) [Results in
  Mathematics and Related Areas (3)], vol.~21, Springer-Verlag, Berlin, 1990.
  \MR{1045822 (91i:14034)}

\bibitem{bookCH}
P.~E. Conner and J.~Hurrelbrink, \emph{Class number parity}, Series in Pure
  Mathematics, vol.~8, World Scientific Publishing Co., Singapore, 1988.
  \MR{963648}

\bibitem{conraddifferent}
Keith Conrad, \emph{The different ideal},
  https://kconrad.math.uconn.edu/blurbs/gradnumthy/different.pdf.

\bibitem{deligne-mostow}
P.~Deligne and G.~D. Mostow, \emph{Monodromy of hypergeometric functions and
  nonlattice integral monodromy}, Inst. Hautes \'Etudes Sci. Publ. Math.
  (1986), no.~63, 5--89. \MR{849651}

\bibitem{deligne79}
Pierre Deligne, \emph{Vari\'{e}t\'{e}s de {S}himura: interpr\'{e}tation
  modulaire, et techniques de construction de mod\`eles canoniques},
  Automorphic forms, representations and {$L$}-functions ({P}roc. {S}ympos.
  {P}ure {M}ath., {O}regon {S}tate {U}niv., {C}orvallis, {O}re., 1977), {P}art
  2, Proc. Sympos. Pure Math., XXXIII, Amer. Math. Soc., Providence, R.I.,
  1979, pp.~247--289. \MR{546620}

\bibitem{KRinv}
Neal Koblitz and David Rohrlich, \emph{Simple factors in the {J}acobian of a
  {F}ermat curve}, Canadian J. Math. \textbf{30} (1978), no.~6, 1183--1205.
  \MR{511556}

\bibitem{kottwitz92}
Robert~E. Kottwitz, \emph{Points on some {S}himura varieties over finite
  fields}, J. Amer. Math. Soc. \textbf{5} (1992), no.~2, 373--444. \MR{1124982}

\bibitem{lan}
Kai-Wen Lan, \emph{Arithmetic compactifications of {PEL}-type {S}himura
  varieties}, London Mathematical Society Monographs Series, vol.~36, Princeton
  University Press, Princeton, NJ, 2013. \MR{3186092}

\bibitem{LangCM}
Serge Lang, \emph{Complex multiplication}, Grundlehren der Mathematischen
  Wissenschaften [Fundamental Principles of Mathematical Sciences], vol. 255,
  Springer-Verlag, New York, 1983. \MR{713612}

\bibitem{LMPT3}
Wanlin Li, Elena Mantovan, Rachel Pries, and Yunqing Tang, \emph{Newton polygon
  stratification of the {T}orelli locus in {PEL}-type {S}himura varieties}, to
  appear in International Math Research Notices, available on arXiv:1811.00604.

\bibitem{LMPT1}
\bysame, \emph{Newton polygons of cyclic covers of the projective line branched
  at three points}, Research Directions in Number Theory: Women in Numbers IV,
  AWM series, pages 115-132, available on arXiv:1805.04598.

\bibitem{LMPT2}
\bysame, \emph{Newton polygons arising from special families of cyclic covers
  of the projective line}, Res. Number Theory \textbf{5} (2019), no.~1, Art.
  12, 31. \MR{3897613}

\bibitem{milneShimura}
J.~S. Milne, \emph{Shimura varieties and moduli}, Handbook of moduli. {V}ol.
  {II}, Adv. Lect. Math. (ALM), vol.~25, Int. Press, Somerville, MA, 2013,
  pp.~467--548. \MR{3184183}

\bibitem{moonenLinearity}
Ben Moonen, \emph{Linearity properties of {S}himura varieties. {I}}, J.
  Algebraic Geom. \textbf{7} (1998), no.~3, 539--567. \MR{1618140}

\bibitem{moonen}
\bysame, \emph{Special subvarieties arising from families of cyclic covers of
  the projective line}, Doc. Math. \textbf{15} (2010), 793--819. \MR{2735989}

\bibitem{shimuraanalytic}
Goro Shimura, \emph{On analytic families of polarized abelian varieties and
  automorphic functions}, Ann. of Math. (2) \textbf{78} (1963), 149--192.
  \MR{156001}

\bibitem{shimuraunitary}
\bysame, \emph{Arithmetic of unitary groups}, Ann. of Math. (2) \textbf{79}
  (1964), 369--409. \MR{158882}

\bibitem{shimuratranscend}
\bysame, \emph{On purely transcendental fields automorphic functions of several
  variable}, Osaka Math. J. \textbf{1} (1964), no.~1, 1--14. \MR{176113}

\bibitem{vanwamelen}
Paul van Wamelen, \emph{Examples of genus two {CM} curves defined over the
  rationals}, Math. Comp. \textbf{68} (1999), no.~225, 307--320. \MR{1609658}

\bibitem{washington}
Lawrence~C. Washington, \emph{Introduction to cyclotomic fields}, second ed.,
  Graduate Texts in Mathematics, vol.~83, Springer-Verlag, New York, 1997.
  \MR{1421575}

\end{thebibliography}

\end{document}